\documentclass[]{article}

\usepackage[left=3cm, right=3.5cm, top=3cm]{geometry}
\usepackage{amssymb}
\usepackage{amsmath}
\usepackage{amsfonts}
\usepackage{amsthm}
\usepackage{mathtools}
\usepackage{listings}
\usepackage{latexsym}
\usepackage{booktabs}
\usepackage{multirow}
\usepackage{graphicx}
\usepackage{xcolor}
\usepackage{epsfig}
\usepackage{epstopdf}
\usepackage{url}
\usepackage{relsize}
\usepackage{hyperref}
\usepackage{textcomp}
\usepackage{indentfirst}
\usepackage{tipa}
\usepackage{subcaption}
\usepackage{dsfont}
\usepackage{float}
\usepackage[mathscr]{euscript}
\usepackage{enumitem}
\usepackage[title]{appendix}
\usepackage{bm}
\usepackage[backend=bibtex,style=numeric,citestyle=numeric,giveninits=true,sortcites=true,maxbibnames=2]{biblatex}

\usepackage{filecontents}

	\immediate\write18{bibtex \jobname}
\usepackage{color}
\definecolor{deepblue}{rgb}{0,0,0.5}
\definecolor{deepred}{rgb}{0.6,0,0}
\definecolor{deepgreen}{rgb}{0,0.5,0}

	\definecolor{DarkBlue}{rgb}{0.00,0.00,0.55}
	\definecolor{Black}{rgb}{0.00,0.00,0.00}
	\hypersetup{
		linkcolor = DarkBlue,
		anchorcolor = DarkBlue,
		citecolor = DarkBlue,
		filecolor = DarkBlue,
		colorlinks  = true,
		urlcolor = Black
	}

\AtBeginDocument{
	\label{CorrectFirstPageLabel}
	
}
	
\graphicspath{ {./Figures/} }
\DeclareGraphicsExtensions{-eps-converted-to.pdf,.pdf,.png,.jpg}  
	
\newtheorem{theorem}{Theorem}[section]
\newtheorem{lemma}[theorem]{Lemma}

\theoremstyle{definition}
\newtheorem{definition}{Definition}[section]

\theoremstyle{remark}
\newtheorem{remark}{Remark}[section]

\newcommand{\TheTitle}{Effects of round-to-nearest and stochastic rounding in the numerical solution of the heat equation in low precision}
 
\newcommand{\TheAuthors}{M.~Croci, M.~B.~Giles}

\title{{\TheTitle}\thanks{\textbf{Funding:} This research is supported by the ICONIC
EPSRC Programme Grant (EP/P020720/1).}}

\author{
  M. Croci\thanks{Oden Institute, University of Texas at Austin, Austin, TX, USA, (\textbf{\url{matteo.croci@austin.utexas.edu}}).}
  and
  M.~B.~Giles\thanks{Mathematical Institute, University of Oxford, Oxford, UK, (\textbf{\url{mike.giles@maths.ox.ac.uk}}).}
}

\usepackage{amsopn}

\DeclareMathOperator{\E}{\mathbb{E}}

\DeclareMathOperator{\bi}{\textbf{\textit{i}}}
\DeclareMathOperator{\EPS}{\mathlarger{\mathlarger{\mathlarger{\varepsilon}}}}
\DeclareMathOperator{\eps}{\bm{\varepsilon}}

\DeclareMathOperator{\F}{\mathcal{F}}
\DeclareMathOperator{\fl}{\text{fl}}
\DeclareMathOperator{\sr}{\text{sr}}
\DeclareMathOperator{\uu}{\mathfrak{u}}
\DeclareMathOperator{\ups}{\nu}
\renewcommand{\P}{\mathbb{P}}

\usepackage{todonotes}
\definecolor{myblue}{RGB}{135, 206, 250}

\newcommand{\R}{\mathbb{R}}
\newcommand{\N}{\mathbb{N}}

\ifpdf
\hypersetup{
  pdftitle={\TheTitle},
  pdfauthor={\TheAuthors}
}
\fi

\begin{document}

\maketitle

\begin{abstract}
  Motivated by the advent of machine learning, the last few years have seen the return of hardware-supported low-precision computing. Computations with fewer digits are faster and more memory and energy efficient, but can be extremely susceptible to rounding errors. As shown by recent studies into reduced-precision climate simulations, an application that can largely benefit from the advantages of low-precision computing is the numerical solution of partial differential equations (PDEs). However, a careful implementation and rounding error analysis are required to ensure that sensible results can still be obtained.
  
  In this paper we study the accumulation of rounding errors in the solution of the heat equation, a proxy for parabolic PDEs, via Runge-Kutta finite difference methods using round-to-nearest (RtN) and stochastic rounding (SR). We demonstrate how to implement the scheme to reduce rounding errors and we derive \emph{a priori} estimates for local and global rounding errors. Let $u$ be the unit roundoff. While the worst-case local errors are $O(u)$ with respect to the discretization parameters (mesh size and timestep), the RtN and SR error behavior is substantially different. In fact, the RtN solution always stagnates for small enough $\Delta t$, and until stagnation the global error grows like $O(u\Delta t^{-1})$. In contrast, we show that the leading-order errors introduced by SR are zero-mean, independent in space and mean-independent in time, making SR resilient to stagnation and rounding error accumulation. In fact, we prove that for SR the global rounding errors are only $O(u\Delta t^{-1/4})$ in 1D and are essentially bounded (up to logarithmic factors) in higher dimensions.
\end{abstract}



\section{Introduction} With the increase of machine learning and GPU acceleration applications as a driving force, the last few years have seen the rise of hardware-supported low-precision computing. In fact, support for the \cite{IEEE754-2008} standard fp16 format and the Google brain floating-point format bfloat16 support is growing among major hardware manufacturers\footnote{The following information is currently available on the manufacturers' respective websites:\\The fp16 format is currently supported by AMD (Instinct) and NVIDIA (Pascal, Volta and Ampere) GPUs, ARM (Armv8-A Neon) chips and partially by Intel (Xeon) chips. The bfloat16 format is supported by Intel NNP processors and FPGAs, Google TPUs, and by ARM (Armv8-A) chips. AMD has added support for bfloat16 in its ROCm libraries and Intel has announced that the future Cooper Lake and Sapphire Rapids CPUs will also support the format.}. These $16$-bit formats allow for much faster operations with much lower memory, bandwidth, and energy requirements in exchange for a reduction in precision and/or range and a consequential increase in rounding error. In this context, the much larger unit roundoff makes traditional worst-case rounding error bounds too large and thus uninformative, and a na\"ive implementation can cause numerical algorithms to lose all accuracy.

Nevertheless, careful algorithmic design can still yield extremely fast and relatively accurate results. For this reason, reduced- or mixed-precision algorithms have become very popular in numerical linear algebra \cite[]{abdelfattah2021survey,HighamMary2019} and machine learning \cite[]{GuptaEtAl2015,NaEtAl2017,OrtizEtAl2018,WangEtAl2018,MellempudiEtAl2019} communities. A substantial research effort is also currently spent to reduce the precision in climate and weather model simulations without affecting accuracy. The results are very promising and show large performance gains with negligible accuracy loss \cite[]{vavna2017single,duben2017study,Klower2021fluid16bit}, and show how numerical algorithms for the solution of PDEs can significantly benefit from a clever reduced-precision implementation.

In this paper we study the accumulation of rounding errors in the solution of the heat equation in $d\in\{1,2,3\}$ dimensions using either round-to-nearest (RtN) or stochastic rounding (SR).
SR was first conceptualized by \cite{Forsythe1950SR,Forsythe1959SR}, and by \cite{HullSwenson1966} and consists in randomly rounding a number $x$ in the floating-point range towards minus or plus infinity with different probabilities so that the resulting rounded quantity, $\sr(x)$, ($\sr(x)$ is the SR function, to be defined later in Section \ref{sec:preliminaries}) is exact in expectation, i.e.~$\E[\sr(x)]=x$. SR has two disadvantages with respect to RtN: since rounding \emph{away} from the nearest representable number is possible, it can produce up to twice the error. Furthermore, it requires random number generation, which is likely to be more expensive. Nevertheless, the use of SR comes with its advantages. For an overview of SR hardware and software implementations, analysis, and applications we refer the reader to the review article by \cite{croci2022stochastic}.

First, the roundoff errors generated by stochastic rounding are random variables, making SR the first rounding mode to be provably amenable to probabilistic rounding error analysis \cite[]{ConnollyHighamMary2020}. In contrast, the probabilistic treatment of deterministic rounding errors is only a modelling device \cite[]{VonNeumann1947,Henrici1962,Henrici1963,HullSwenson1966,HighamMary2019} and the smaller probabilistic bounds might not hold for RtN in some scenarios where systematic accumulation of rounding errors occurs \cite[]{HighamMary2019}.

Second, SR is effective at avoiding stagnation, a phenomenon that occurs e.g. when a (relatively) small, exactly representable number $\epsilon$ is added to a large, exactly representable number $x$. In RtN mode, if $\epsilon$ is smaller than half the gap between $x$ and its representable neighbours, then the quantity $x+\epsilon$ is rounded back to $x$ and the term $\epsilon$ is neglected. With SR this phenomenon does not occur since $\E[\sr(x+\epsilon)]=x+\epsilon$ and the addition is not ignored. 

For both of these reasons, SR in low precision has become popular in the machine learning literature since it can yield higher accuracy results at a lower cost when training neural networks \cite[]{GuptaEtAl2015,EssamEtAl2017,NaEtAl2017,OrtizEtAl2018,WangEtAl2018,MellempudiEtAl2019}. The effect of SR in the solution of ordinary and partial differential equations (ODEs and PDEs), and neural ODEs in low precisions has also been investigated experimentally \cite[]{Hopkins2019,FasiMikaitis2020,paxton2021climate}, showing that SR can still yield accurate results in situations where RtN fails to produce meaningful results. SR is also used in neuromorphic computing \cite[]{Loihi2018,SpiNNaker2}.

While SR resilience to stagnation is now well understood \cite[]{ConnollyHighamMary2020}, only limited theoretical results are present in the literature about the behavior of SR in common numerical algorithms. The only theoretical work that we are aware of is the recent paper by \cite{ConnollyHighamMary2020} in which the authors study the effects of SR in linear algebra operations and show how SR fits within recent research on probabilistic rounding error analysis \cite[]{HighamMary2019,IpsenZhou2019,HighamMary2020}. In the context of ODEs, the only applicable results that we are aware of come from the 1960s and the classical work of \cite{Henrici1962,Henrici1963} and the more recent paper by \cite{Arato1983}. While SR is not mentioned, Henrici and Arat\'o build on the assumption that at each timestep a local rounding error is introduced, which is treated as an independent, almost surely (a.s.) bounded random variable. A similar assumption is used by \cite{Jezequel1995}, who derives estimates for the first and second moment of the componentwise error caused by deterministic rounding modes in the solution of the heat equation. However, the results in \cite{Jezequel1995} do not improve on the worst-case linear error growth rate that is typical of  deterministic analyses of ODE solvers \cite[]{Henrici1962,Henrici1963}. We are not aware of any study in the literature in which sharper probabilistic rounding error estimates are derived in a context specific to PDEs.

It is worth noting that SR hardware support is currently scarce, yet growing due to its popularity in the machine learning and artificial intelligence community. SR is currently used by Intel neuromorphic chip Loihi \cite[]{Loihi2018} and will be supported in the future neuromorphic supercomputing architecture SpiNNaker2 \cite[]{SpiNNaker2}. The new Graphcore Intelligence Processing Unit (IPU), designed for machine learning and artificial intelligence applications, currently supports SR \cite[]{Graphcore2019}. Patents by Intel \cite[]{p_KaulEtAl2018}, AMD \cite[]{p_Loh2019}, NVIDIA \cite[]{p_AlbenEtAl2020}, IBM \cite[]{p_BradburyEtAl2017,p_Bradbury2017}, and other companies \cite[]{p_Kanter2010,p_Henry2018,p_Lifsches2020} have been filed describing possible SR implementations, suggesting that more chips supporting SR might be in the works. Finally, \cite{HighamPranesh2019} have developed a low precision emulator which supports stochastic rounding ``\emph{because of its increasing adoption in machine learning and the need for better understanding of its behavior}''. Other low-precision emulators supporting SR are available (cf.~\cite{fousse2007mpfr,zhang2019qpytorch,fasi2021algorithms}), see Section 7(d) in \cite{croci2022stochastic} for an overview.

\vspace{3pt}
In this paper, we make the following new contributions:
\begin{itemize}
	\item We analyze how rounding errors affect the numerical solution of the heat equation in 1D, 2D and 3D on the unit box. We consider a standard second-order finite-difference scheme in space and an arbitrary Runge-Kutta method in time, and we show how the numerical scheme should be implemented to reduce the effect of rounding errors and fully exploit the resilience to stagnation of SR. The key element is the use of the delta form of the timestepping scheme. While the delta form is a well-known technique in the numerical analysis and computational fluid dynamics community (cf.~Section 12.6 in \cite{lomax2013fundamentals}, and Chapter IV.8 in \cite{HairerWanner1996}), we are unaware of our rigorous analysis of its benefits. Our theory and experiments establish and confirm its efficacy in a probabilistic setting.
	
	\item Let $u$ be the unit roundoff. We show that local rounding errors (introduced at each timestep) are essentially $O(u)$ with respect to the discretization parameters and are neither independent across timesteps, nor can they be modelled as such. Previous work on probabilistic rounding error analysis for vector ODEs \cite[]{Henrici1962,Henrici1963,Arato1983} thus cannot be generalized to the parabolic PDE case since it relies on the local rounding errors to be $O(u\Delta t^p)$ for some $p\geq2$ and independent across timesteps.
	
	\item We derive \emph{a priori} error bounds for both local and global rounding errors. The local error estimates are worst-case bounds and hold for both RtN and SR. The global rounding error estimates depend on the dimension and are specific to the rounding mode used. In the latter case strong error bounds for the expected discrete $L^2$ and $L^\infty$ norms are provided. The key new result is that while it is known that for RtN the global error grows like $O(u\Delta t^{-1})$ in $d$-dimensions \cite[]{Henrici1962,Henrici1963}, for SR it is only $O(u\Delta t^{-1/4})$ in 1D and essentially $O(u)$ (up to logarithmic factors) for $d>1$. This is a much stronger result than what is achieved for numerical linear algebra computations, where SR was shown to only reduce the global rounding error from $O(nu)$ to $O(\sqrt{n}u)$, ($n$ here is the problem size) \cite[]{ConnollyHighamMary2020}. In contrast with numerical linear algebra, when solving the heat equation SR is therefore sufficient to achieve very good accuracy without resorting to other rounding error reducing techniques such as compensated summation; cf.~\cite[Section 4.3]{higham2002accuracy}. Numerical results support our findings.
\end{itemize}

This paper is structured as follows. In Section \ref{sec:preliminaries} we give some background information, we describe some basic properties of SR and we introduce the numerical scheme used. In Section \ref{sec:local} we describe how the numerical scheme should be implemented to reduce rounding errors and we derive worst-case local rounding error bounds for both the ``na\"ive'' and more accurate implementation. We conclude the section by showing that with RtN the numerical scheme always stagnates for small enough $\Delta t$. In Section \ref{sec:global} we combine the local results into global estimates for the global rounding error. We support our findings with numerical results presented in Section \ref{sec:num_res} and we give some concluding remarks in Section \ref{sec:conclusions}.

\section{Preliminaries}
\label{sec:preliminaries}

\subsection{Finite precision arithmetic and types of rounding}
In this paper we only consider the base-$2$ normalized floating-point number systems presented in Table \ref{tab:precision}. We define a floating-point number system as in the book by \cite{higham2002accuracy}:
\begin{definition}[Section 2.1 in \cite{higham2002accuracy}]
	A floating-point number system $F \subset \R$ is a subset of the real numbers whose elements have the form
	\begin{align}
	\label{eq:2}
	y=\pm m \times \beta^{e-t},
	\end{align}
	where $\beta,m,t\in\N$ and $e\in\mathbb{Z}$. The quantity $\beta$ is the base ($\beta=2$ in this paper), $t$ is the precision, $e$ with $e_{\min} \leq e \leq e_{\max}$ is the exponent and $m$ with $0\leq m \leq \beta^t -1$ is the significand. 
\end{definition}

\begin{table}[h!]
	\centering
	\begin{tabular}{@{}llllcc@{}}
		\toprule
		\multicolumn{1}{c}{Format} & \multicolumn{1}{c}{$u$} & \multicolumn{1}{c}{$x_{\min}$} & \multicolumn{1}{c}{$x_{\max}$} & $t$ & exponent bits \\ \midrule
		bfloat16                   & $3.91\times 10^{-3}$    & $1.18\times 10^{-38}$          & $3.39\times 10^{38}$           & $8$         & $8$      \\
		fp16                       & $4.88\times 10^{-4}$    & $6.10\times 10^{-5}$           & $6.55\times 10^{4}$            & $11$        & $5$      \\
		fp32 (single)              & $5.96\times 10^{-8}$    & $1.18\times 10^{-38}$          & $3.40\times 10^{38}$           & $24$        & $8$      \\
		fp64 (double)              & $1.11 \times 10^{-16}$  & $2.22\times 10^{-308}$         & $1.80\times 10^{308}$          & $53$        & $11$     \\ \bottomrule
	\end{tabular}
	\caption{\textit{Parameters of the floating-point number systems mentioned in the paper. Here $u=2^{-t}$ is the unit roundoff, $x_{\min}$ in the smallest normalized positive number, $x_{\max}$ is the largest finite number and $t$ is the precision. Note that bfloat16 and fp32 have the same exponent bits.}}
	\label{tab:precision}
	\vspace{0pt}
\end{table}
\begin{remark}
	For the sake of simplicity, the theory presented in this paper ignores effects related to the range of $F$ (such as underflow/overflow). The numerical results of Section \ref{sec:num_res}, however, do account for such behavior.
\end{remark}

A rounding function is a function $\fl:\R\rightarrow F$. In what follows, we work under the following standard floating-point error model (cf.~Chapter 2 of \cite{higham2002accuracy}):
\begin{align}
\label{eq:std_floating_point_error_model}
\text{fl}(a \text{ op } b) = (a \text{ op } b)(1 + \delta),\quad |\delta| < 2u,\quad\text{op}\in\{+,-,\times,\backslash\},
\end{align}
where $u$ is the unit roundoff and the factor of $2$ is to account for any type of rounding (not just round-to-nearest (RtN)). In this paper, we denote with $\fl$ the standard RtN function and with $\sr$ the SR function, defined as follows.
\begin{definition}[Stochastic rounding function]
	Let $x\in\R$ and $a,b\in F$ with $a\leq x\leq b$ to be the numbers of $F$ that are adjacent to $x$. Set $I=b-a$, Then the SR function $\sr(x)$ is the function that rounds $x$ to $a$ with probability $\varrho = (b-x)/I$ and to $b$ with probability $1-\varrho=(x-a)/I$.
\end{definition}

The following result, proven by \cite{ConnollyHighamMary2020}, shows that the roundoff errors arising from SR are mean-independent zero-mean random variable whose absolute value is bounded by $2u$.

\begin{lemma}[cf.~Lemmas 4.4 and 4.5 in \cite{ConnollyHighamMary2020}]
	\label{lemma:sr_properties}
	The SR function can be written as $\sr(x)=x(1+\delta)$, where $|\delta| < 2u$. Furthermore, the roundoff errors generated by SR are mean-independent zero-mean random variables, i.e.~they satisfy
	\begin{align}
	\label{eq:1}
	\E[\delta_i | \F_{i-1}] = \E[\delta_i] = 0,\quad\forall i,\quad\text{where}\quad \F_{i}=\{\delta_1,\dots,\delta_i\},
	\end{align}
	assuming that $\delta_1,\dots,\delta_i$ are generated in this order.
\end{lemma}

\begin{remark}
	To simplify the notation, from now on we set $\ups=u$ for RtN and $\ups=2u$ for SR so that we can just write $\ups$ when presenting results that apply to both. Furthermore, we use the following bounds valid for $n\ups<1$ \cite[Lemma 3.1]{higham2002accuracy},
	\begin{align*}
	\prod_{i=1}^m(1+\delta_i)^{\rho_i}=1 + \theta_m,\quad\text{with}\quad |\theta_m| \leq \gamma_m\ \text{(RtN and SR)},\quad |\theta_m| \leq \tilde{\gamma}_m\ \text{ with high probability }\ \text{(SR only)}
	\end{align*}
	where $\delta_i$ are roundoff errors generated with the same rounding mode, $\rho_i=\pm 1$ and
	\begin{align*}
	\gamma_m=\frac{m\ups}{1-m\ups},\quad \tilde{\gamma}_m=\exp\left(\frac{2c\sqrt{m}u+4mu^2}{1-2u}\right)-1=2c\sqrt{m}u+O(u^2).
	\end{align*}
	Here $c>0$ is an arbitrary constant that can typically be taken to be $O(1)$ \cite[]{ConnollyHighamMary2020}. The $\gamma_m$ bound is a classical result which always holds for both RtN and SR \cite[]{higham2002accuracy}, while the $\tilde{\gamma}_m$ bound for SR is a stochastic bound (for SR $\theta_m$ is a random variable) that was proved recently \cite[]{ConnollyHighamMary2020} and it only holds with high probability. In what follows we only use $\gamma_m$ for simplicity, but it can be replaced with $\tilde{\gamma}_m$ in the SR case, yielding smaller constants. 
\end{remark}

\subsection{Numerical discretization of the heat equation}

In this paper, we investigate the effects of rounding errors in the numerical solution of the heat equation with non-zero forcing term:

\begin{align}
\label{eq:main_eqn}
\left\{\begin{array}{llc}
\dot{\uu}(t,\bm{x}) - \Delta \uu(t,\bm{x}) = f(t,\bm{x}), & \bm{x}\in D=(0,1)^d, & t\in(0,T],\\
\uu(0,\bm{x}) = \uu_0(\bm{x}) & \bm{x}\in D=(0,1)^d, & \\
\uu(t,\bm{x}) = g(\bm{x}) & \bm{x} \in \partial D, & t\in(0,T],
\end{array}\right.
\end{align}
where $0\not\equiv f\in C^0(\bar{D}\times [0,T])$ and $\uu_0 \in C^2(\bar{D})$ with $\uu_0|_{\partial D} = g$. We choose time-independent boundary conditions for simplicity.

Let $\bi\in\N^d$ be a multi-index. We first discretize \eqref{eq:main_eqn} in space at $K+1$ discrete equispaced points $\{\bm{x}_{\bi}\}$ in each spatial direction with $\bm{x}_{\bi}=\bi h$ and $h=1/K$ using the standard finite difference scheme
\begin{align}
\label{eq:main_eqn_FD}
\dot{U}_{\bi}(t) - h^{-2}\sum\limits_{j=1}^d (U_{\bi+\bm{e}_j}(t) - 2U_{\bi}(t) + U_{\bi-\bm{e}_j}(t)) = f_{\bi}(t),\quad t\in[0,T],
\end{align}
where $U_{\bi}(t) \approx \uu(t,\bm{x}_{\bi})$, $f_{\bi}(t) = f(t,\bm{x}_{\bi})$, and $\bm{e}_j\in\R^d$ for $j=1,\dots,d$ are the canonical basis vectors. This can be rewritten in matrix form as:
\begin{align}
\label{eq:main_eqn_discrete}
\dot{\bm{U}}(t) + A \bm{U}(t) = \bm{f}(t),\quad t\in(0,T],
\end{align}
where $A\in\R^{(K-1)^d \times (K-1)^d}$ is the standard finite difference matrix of the $d$-di\-men\-sio\-nal Laplacian and $\bm{f}(t)$ incorporates the boundary conditions.

The eigenvalues and eigenfunctions of the continuous operator $-\Delta$ are \cite[Chapter VI]{CourantHilbert1953}
\begin{align}
\lambda_{\bm{k}}=\pi^2 \Vert \bm{k}\Vert_2^2,\quad v_{\bm{k}}(\bm{x})=\prod\limits_{j=1}^d\left(\sqrt{2}\sin(\pi\bm{k}_jx_j)\right),
\end{align}
where $\bm{k}\in\N_0^d$ is a multi-index. Their discrete version, the eigenpairs of $A$, are \cite[]{Kuttler1974}
\begin{align}
\label{eq:eigenpairsA}
\lambda_{\bm{k}}^h=4 K^2\sum\limits_{j=1}^d\sin^2\left(\frac{\pi}{2} \frac{\bm{k}_j}{K}\right),\quad \bm{v}^h_{\bm{k},\bi}=\prod\limits_{j=1}^d\left(\sin\left(\pi\bi_j\frac{\bm{k}_j}{K}\right)\right).
\end{align}
Here we have indexed the $\bm{k}$-th eigenvector $\bm{v_k}^h$ of $A$ as if it were a $d$-dimensional tensor with entries indexed by the multi-index $\bm{i}\in\N^d_0$. Both $\bm{i}$ and $\bm{k}$ satisfy $\bm{1}\leq \bi,\bm{k} \leq (K-1)\bm{1}$ entrywise, where $\bm{1}\in \R^d$ is the vector of all ones. Independently from the notation adopted, we will always consider $\bm{v}_{\bm{k}}$ to be a vector in $\R^{(K-1)^d}$. Note that we picked the scaling of the eigenvectors so that
\begin{align}
\label{eq:eigenpairsA_bounds}
\Vert\bm{v}^h_{\bm{k}}\Vert_2^2 = \left(\frac{K}{2}\right)^d,\quad \Vert\bm{v}^h_{\bm{k}}\Vert_\infty \leq 1,\quad\forall\bm{k},
\end{align}
where the norms are to be intended as vector norms. The discrete eigenvectors are orthogonal so that if $\bm{k}\neq\tilde{\bm{k}}$ or if $\bi\neq\tilde{\bi}$ we have that $(\bm{v}^h_{\bm{k},\bm{j}},\bm{v}^h_{\tilde{\bm{k}},\bm{j}}) = (\bm{v}^h_{\bm{k},\bi},\bm{v}^h_{\bm{k},\tilde{\bi}}) = 0$. Here $(\cdot,\cdot)$ denotes the vector inner product. 

By expanding $\bm{U}(t)$ in terms of the eigenvectors of $A$ we can decouple the ODE system \eqref{eq:main_eqn_discrete} into many scalar ODEs: set $\bm{U}(t)=\sum_{\bm{k}=\bm{1}}^{\bm{k}=(K-1)\bm{1}}a_{\bm{k}}(t)\bm{v}^h_{\bm{k}}$, and $\bm{f}(t)=\sum_{\bm{k}=\bm{1}}^{\bm{k}=(K-1)\bm{1}}f_{\bm{k}}(t)\bm{v}^h_{\bm{k}}$, plug these into \eqref{eq:main_eqn_discrete} and multiply by $\bm{v}_{\bar{\bm{k}}}$ to obtain
\begin{align}
\label{eq:eig_decoupled_ode}
\dot{a}_{\bar{\bm{k}}}(t) + \lambda^h_{\bar{\bm{k}}}a_{\bar{\bm{k}}}(t) = f_{\bar{\bm{k}}}(t),\quad t\in(0,T].
\end{align}

We now discretize \eqref{eq:main_eqn_discrete} in time by applying a generic Runge-Kutta (RK) method with absolute stability function $S(z)$, $z\in\mathbb{C}$; see \cite[Chapter 3]{ButcherBook2016}. Note that for explicit RK methods $S(z)$ is always a polynomial while for implicit methods $S(z)$ is a rational function. In both cases, $S(z)$ can always be written in the form $S(z) = 1 + z\tilde{S}(z)$ where $\tilde{S}(0)=1$; see \cite[Lemma 351A]{ButcherBook2016}. We discretize $\bm{U}(t)$ at times $t^n=n\Delta t$ for $n=0, \dots, N=T/\Delta t$ (assuming $T/\Delta t$ is integer). By applying the RK scheme to \eqref{eq:eig_decoupled_ode} and by summing over all eigenmodes we obtain the following scheme,
\begin{align}
\label{eq:num_scheme}
\bm{U}^{n+1} = S(-\Delta tA)\bm{U}^{n} + \Delta t \bm{F}^n &=\bm{U}^{n} + \left(-\Delta t\tilde{S}(-\Delta tA)A\bm{U}^{n} + \Delta t \bm{F}^n\right) \\
&= \bm{U}^{n} + (\Delta \bm{U})^{n},\notag\\
\text{where}\quad\bm{F}^n=\sum_{j=1}^sQ_j(-\Delta t A)\bm{f}^n_j&=\sum_{j=1}^sQ_j(-\Delta t A)\bm{f}(t^n+c_j\Delta t).
\end{align}
Here, all $Q_j(z)$ are rational functions (or polynomials if the method is explicit) that satisfy $\sum_{j=1}^sQ_j(z)=\tilde{S}(z)$. The $c_j\in \R$ are the nodes of the RK method and $s\in\N_0$ is the number of stages of the method, e.g.~see \cite[Chapter 3]{ButcherBook2016}.
The two equations on the right in \eqref{eq:num_scheme} are the \emph{delta form} of the scheme \cite[Section 12.6]{lomax2013fundamentals}, \cite[Chapter IV.8]{HairerWanner1996}, in which $(\Delta \bm{U})^n$ is computed first and then added to $\bm{U}^n$. As we will see in the next section, evaluating the scheme in its delta form helps to reduce rounding errors.

\section{Local rounding error analysis}
\label{sec:local}

When working in finite precision, the numerical scheme \eqref{eq:num_scheme} cannot be evaluated exactly. What we instead compute is
\begin{align}
\label{eq:num_scheme_finite_precision}
\widehat{\bm{U}}^{n+1} = S(-\Delta tA)\widehat{\bm{U}}^{n} + \Delta t\bm{F}^n + \mathlarger{\mathlarger{{\varepsilon}}}^n,
\end{align}
where $\EPS^n\in\R^{(K-1)^d}$ is the local rounding error and it incorporates all rounding errors introduced at timestep $n$ for a generic implementation. When the numerical scheme is implemented in delta form, we write instead
\begin{align}
\label{eq:num_scheme_finite_precision_delta}
\widehat{\bm{U}}^{n+1} = \widehat{\bm{U}}^{n} + (\Delta\widehat{\bm{U}})^n + \eps^n +~\Theta^n.
\end{align}
To properly define $\Theta^n$ and $\eps^n$, let $\tilde{\Theta}^n$ be the local rounding error arising from the computation of $(\Delta\widehat{\bm{U}})^n$, then we have under the standard floating-point error model \eqref{eq:std_floating_point_error_model},
\begin{align}
&\widehat{\bm{U}}^{n+1}_{\bm{j}} - \left(\widehat{\bm{U}}^{n}_{\bm{j}} + (\Delta\widehat{\bm{U}})^n_{\bm{j}}\right) = (1+\delta_{\bm{j}})\left(\widehat{\bm{U}}^{n}_{\bm{j}} + ((\Delta\widehat{\bm{U}})^n_{\bm{j}} + \tilde{\Theta}^n_{\bm{j}})\right)- \left(\widehat{\bm{U}}^{n}_{\bm{j}} + (\Delta\widehat{\bm{U}})^n_{\bm{j}}\right)\notag\\
=&\left(\tilde{\Theta}^n_{\bm{j}} + \delta_{\bm{j}}(\Delta\widehat{\bm{U}})^n_{\bm{j}} + \delta_{\bm{j}}\tilde{\Theta}^n_{\bm{j}}(\Delta\widehat{\bm{U}})^n_{\bm{j}}\right) + \delta_{\bm{j}}\widehat{\bm{U}}^{n}_{\bm{j}} = \left(\Theta^n_{\bm{j}}\right) + \eps^n_{\bm{j}},
\label{eq:Theta}
\end{align}
where $\eps^n_{\bm{j}} = \delta_{\bm{j}}\widehat{\bm{U}}^{n}_{\bm{j}}$ essentially captures the leading order error term arising from the addition of $\widehat{\bm{U}}^{n}$ and $(\Delta\widehat{\bm{U}})^n$. The distinction between $\eps^n$ and $\Theta^n$ is important: in Sections \ref{sec:global} and \ref{sec:num_res} we show how the $\eps^n$ term dominates and is much easier to analyse than $\Theta^n$. Note that $\EPS^n$ is typically different from $\eps^n + \  \Theta^n$ since the errors come from different implementations. In fact, we show in Section \ref{sec:num_res} that the delta form local error is typically smaller.

Our final aim is to derive an estimate for the global rounding error, defined as a norm of the quantity
\begin{align}
\bm{E}^n = \widehat{\bm{U}}^n-\bm{U}^n,
\end{align}
which we will call the global rounding error vector at timestep $n$.

\begin{remark}
	For completeness, we define $\EPS^{-1}$, $\eps^{-1}$ to be the local errors due to the rounding of the initial condition. They satisfy $\EPS^{-1} \equiv \eps^{-1} := \widehat{\bm{U}}^0 - \bm{U}^0$.
\end{remark}

In this section we derive worst-case error bounds for the local errors $\EPS^n$, $\eps^n$ and $\Theta^n$ (valid for both RtN and SR) and we explain why the RtN solution is always affected by stagnation for small enough $\Delta t$. We then use the local error results to obtain a bound for the global rounding error in Section \ref{sec:global}.

\subsection{Matrix-vector products in the delta form implementation}
\label{subsec:matvecs}

When working with the delta form, we assume that the matrix-vector product $-A\widehat{\bm{U}}^n$ is computed via first order differences as follows (using the same notation as in \eqref{eq:main_eqn_FD}).
\begin{align}
\label{eq:2firstorder_matvec}
(-A\widehat{\bm{U}}^n)_{\bi} = h^{-2}\sum\limits_{j=1}^d[(\widehat{\bm{U}}_{\bi+\bm{e}_j}^n - \widehat{\bm{U}}_{\bi}^n) - (\widehat{\bm{U}}_{\bi}^n - \widehat{\bm{U}}_{\bi-\bm{e}_j}^n)].
\end{align}
This approach is much less prone to rounding errors than using a direct second order difference (i.e.~made of terms $\widehat{\bm{U}}_{\bi+\bm{e}_j}^n - 2\widehat{\bm{U}}_{\bi}^n + \widehat{\bm{U}}_{\bi-\bm{e}_j}^n$). The following two theorems explain why. The first theorem is a generalization of Sterbenz lemma (cf.~Theorem 2.5 in \cite[]{higham2002accuracy}):
\begin{theorem}
	\label{th:2firstorderdiff}
	If $a,b,c\in\R$ are exactly
	represented in floating-point arithmetic, and 
	\begin{align}
	\max\left( |a-b|, |b-c|\right) \leq
	\min( |a|, |b|, |c|) 
	\end{align}
	then $a-b$ and $b-c$ are both computed exactly as long as subnormal numbers are supported or underflow does not occur.
%
\end{theorem}

\begin{proof}
	Let $\bar{e}\in\mathbb{Z}$ be the largest integer so that $a$, $b$,
	$c$ can be written as 
	$
	a = k \times 2^{\bar{e}}, ~~ 
	b = l \times 2^{\bar{e}}, ~~
	c = m \times 2^{\bar{e}},
	$
	where $k,l,m\in\mathbb{Z}$.	Hence,
	\begin{align}
	a - b = (k-l) \times 2^{\bar{e}}, ~~
	b - c = (l-m) \times 2^{\bar{e}}.
	\end{align}
	The condition on $a-b$ implies that
	$|k-l| \leq \min(|k|,|l|,|m|)$ and therefore (if subnormal numbers are supported or barring underflow)
	$(k-l) \times 2^{\bar{e}}$ has an exact floating-point representation,
	and the same applies to $(l-m) \times 2^{\bar{e}}$.
\end{proof}

Depending on which one of the three statements in the following theorem holds, we have a rounding error arising from the term $-\Delta t A\widehat{\bm{U}}^n$ of order $O(u\Delta t/h^2)$ (Statement 1), $O(u\Delta t/h)$ (Statement 2) or $O(u\Delta t)$ (Statement 3). In the latter two cases, the error is reduced by an $O(h)$ and $O(h^2)$ factor respectively, and decreases with $\Delta t$, thus explaining why the evaluation of the matrix-vector product via \eqref{eq:2firstorder_matvec} is much more accurate.

\begin{theorem}
	\label{th:2firstorderdiff2}
	Let $\ups<1/4$ and consider the evaluation of
	\begin{align}
	\label{eq:2firstorder_matvec2}
	\sum\limits_{j=1}^d[(\widehat{\bm{U}}_{\bi+\bm{e}_j}^n - \widehat{\bm{U}}_{\bi}^n) - (\widehat{\bm{U}}_{\bi}^n - \widehat{\bm{U}}_{\bi-\bm{e}_j}^n)].
	\end{align}
	The following statements hold:
	\begin{enumerate}
		\item If for all $\bm{i},n$ $|\widehat{\bm{U}}^n_{\bm{i}}|\leq M$, then the evaluation of both \eqref{eq:2firstorder_matvec2} and a direct second-order difference produces a rounding error of at most $4d\gamma_{d+1}M$. Furthermore, \eqref{eq:2firstorder_matvec2} is bounded in absolute value by $4dM$.
		\item If there exists $\bar{c}>0$ such that for all $\bi,j,n$
		\begin{align}
		\label{eq:th2firstorder1}
		\max\left(|\widehat{\bm{U}}_{\bi+\bm{e}_j}^n - \widehat{\bm{U}}_{\bi}^n|,|\widehat{\bm{U}}_{\bi}^n - \widehat{\bm{U}}_{\bi-\bm{e}_j}^n|\right) \leq \bar{c}h,
		\end{align}
		then \eqref{eq:2firstorder_matvec2} is bounded in absolute value by $2\bar{c}dh$ and its evaluation produces a rounding error of at most $2\bar{c}d\gamma_{d+1}h$.
		\item If for all $\bi,j,n$ the assumptions of Theorem \ref{th:2firstorderdiff} hold with $a=\widehat{\bm{U}}_{\bi+\bm{e}_j}^n$, $b=\widehat{\bm{U}}_{\bi}^n$ and $c=\widehat{\bm{U}}_{\bi-\bm{e}_j}^n$, and there exists $\bar{c}>0$ such that
		\begin{align}
		\label{eq:th2firstorder2}
		|(\widehat{\bm{U}}_{\bi+\bm{e}_j}^n - \widehat{\bm{U}}_{\bi}^n) - (\widehat{\bm{U}}_{\bi}^n - \widehat{\bm{U}}_{\bi-\bm{e}_j}^n)| \leq \bar{c}h^2,
		\end{align}
		then \eqref{eq:2firstorder_matvec2} is bounded in absolute value by $\bar{c}dh^2$ and its evaluation produces a rounding error of at most $\bar{c}d\gamma_dh^2$.
	\end{enumerate}
	
\end{theorem}

\begin{proof}
	The upper bounds on the absolute value of the rounding errors arising from the summation in \eqref{eq:2firstorder_matvec2} are obtained via the triangle inequality so we only need to prove bounds on the rounding errors of each of the individual terms.	Evaluation of each term in the sum in \eqref{eq:2firstorder_matvec2} in finite precision yields,
	\begin{align}
	\left((\widehat{\bm{U}}_{\bi+\bm{e}_j}^n - \widehat{\bm{U}}_{\bi}^n)(1+\delta_1) - (\widehat{\bm{U}}_{\bi}^n - \widehat{\bm{U}}_{\bi-\bm{e}_j}^n)(1+\delta_2)\right)(1+\delta_3),
	\end{align}
	where $\delta_1$ and $\delta_2$ are the rounding errors in the first order differences and $\delta_3$ is the roundoff coming from the last subtraction. Note that all entries of $\widehat{\bm{U}}^n$ are already rounded by construction, hence exactly representable. Under the assumptions of Statement 1 we have a maximum rounding error of $4M((1+\ups)^2-1)$ for each term in the sum in \eqref{eq:2firstorder_matvec2}. After accounting for the summation error we obtain a rounding error bound of $4dM((1+\ups)^{d+1}-1)\leq 4d\gamma_{d+1}M$. The same reasoning yields a bound of $2\bar{c}d\gamma_{d+1}h$ if \eqref{eq:th2firstorder1} holds. To prove Statement 3, we note that owing to Theorem \ref{th:2firstorderdiff} the first order differences are computed exactly, so $\delta_1=\delta_2=0$. This gives an error of $\bar{c}\ups h^2$ for each term in the summation and a total error bounded above by $\bar{c}d\gamma_dh^2$, which is the desired result. If we use a second-order difference, we get instead
	\begin{align}
	\left((\widehat{\bm{U}}_{\bi+\bm{e}_j}^n - 2\widehat{\bm{U}}_{\bi}^n)(1+\delta_1) + \widehat{\bm{U}}_{\bi-\bm{e}_j}^n\right)(1+\delta_2),
	\end{align}
	and the only obtainable upper bound for each term in the summation is the one from Statement 1. Here we have assumed that multiplication by $2$ is exact in base-$2$ floating-point arithmetic, although this assumption matters little for the sake of the argument. The result does not change if a different summation order is used for the evaluation. 
\end{proof}

\begin{remark}
	In the limit $h \rightarrow 0$, the assumptions of Statements 2 and 3 in Theorem \ref{th:2firstorderdiff2} do not necessarily hold since they require that in the limit $\widehat{\bm{U}}_{\bi}^n = \widehat{\bm{U}}_{\bi\pm\bm{e}_j}^n$ for all $\bi,j,n$. However, we believe that it is reasonable to assume at least \eqref{eq:th2firstorder1} for $u\ll h$.
\end{remark}

\subsection{Worst-case bounds for the local rounding errors}

In what follows, we prove that the components of $\EPS^n$, $\eps^n$, $\Theta^n$ are bounded so that there exist positive constants $\EPS$, $\varepsilon$, $\theta$ such that $|\EPS_j^n|\leq \EPS$, $|\eps^n_j|\leq \varepsilon$, $|\Theta^n_j|\leq \theta$ for all $j$, $n$. Note that $\EPS$, $\varepsilon$, and $\theta$ typically depend on the rounding mode, on the unit roundoff $u$, on the ratio $\lambda = \Delta t/h^2$, and on the RK scheme used. We show that when the RK method is explicit $\EPS$ and $\varepsilon$  depend on $\lambda$, but not on $\Delta t$ and $K$, while $\theta$ is small for small $\Delta t$. Things behave similarly for implicit methods, but the bounds might also grow with $K$, depending on the linear solver used. We make the following three assumptions:
\begin{enumerate}
	\item[(A1)] There exist constants $C_f,M_f>0$ such that for all $(t,\bm{x})\in[0,T]\times D$, we have $|f(t,\bm{x})|\leq M_f$ and the evaluation of $f(t,\bm{x})$ in finite precision produces
	\begin{align*}
	\widehat{f}(t,\bm{x})=f(t,\bm{x}) + (\Delta f)(t,\bm{x}),\ \text{with}\ |(\Delta f)(t,\bm{x})| \leq C_f \ups|f(t,\bm{x})|\leq C_fM_f\ups.
	\end{align*}
	\item[(A2)] There exists $M>0$ such that the quantities $\widehat{\bm{U}}^{n}$ satisfy $\Vert\widehat{\bm{U}}^{n}\Vert_\infty \leq M$ for all $n$. Note that $M$ might depend on $h$ and $\Delta t$.
	\item[(A3)] We assume that for $p\geq 0$, $\Vert\Delta t A\widehat{\bm{U}}^n\Vert_\infty\leq c_a(d,\lambda)\Delta t^p$ for some constant $c_a(d,\lambda)>0$, and that the evaluation of $-\Delta t A\widehat{\bm{U}}^n$ yields an error bounded above by $c_a(d,\lambda)\Delta t^p\ups$. For instance, this is the case if the hypotheses of Theorem \ref{th:2firstorderdiff2} are satisfied with $p=0$, $p=1/2$ or $p=1$ according to whether Statement 1, 2 or 3 holds.
\end{enumerate}\vspace{3pt}

\noindent
First, we need an auxiliary result:
\begin{theorem}
	\label{th:matrix_rational_function_bound}
	Let $\bm{v}$ be an arbitrary vector in $\R^{(K-1)^d}$. Let $P(x)=N(x)/W(x)$ be either $S(x)$, $\tilde{S}(x)$ or $Q_j(x)$ for any $j=1,\dots,s$. Assume that $P(x)$ has no poles on the negative real axis. Let $\bm{z}$ solve $W(-\Delta t A)\bm{z}=N(-\Delta t A)\bm{v}$ and let $\widehat{\bm{z}}$ be the result of computing $\bm{z}$ in finite precision by solving the linear system using a stable solver. Then, there exist constants $C_0,C>0$ such that
	\begin{align}
	\Vert\bm{z}\Vert_\infty\leq C_0(s,d,\lambda,K)\Vert\bm{v}\Vert_\infty,\quad \Vert\widehat{\bm{z}}-\bm{z}\Vert_{\infty}\leq C(s,d,\lambda,K)\ups\Vert\bm{v}\Vert_\infty.
	\end{align}
	The second bound only holds for $\ups$ sufficiently small. If the RK method is explicit the constants $C_0$ and $C$ still depend on $\lambda$, but not directly on $K$ (recall that $K=1/h$). The bound still holds, albeit with a slightly larger constant, if $\bm{v}$ is also computed inexactly yielding $\widehat{\bm{v}}$ with $\Vert\widehat{\bm{v}} - \bm{v}\Vert_\infty \leq c\ups\Vert\bm{v}\Vert_{\infty}$ for some $c>0$. 
\end{theorem}

\begin{proof}
	This is a combination of classical results from \cite{higham2002accuracy}, cf.~\ref{sec:appA}.
\end{proof}

We recall that requiring that a linear system, e.g.~$W\bm{z}=\bm{y}$, is solved by using a stable solver means that the computed solution $\widehat{\bm{z}}$ satisfies for some $c_1,c_2>0$, (cf.~\cite{higham2002accuracy})
\begin{align}
\label{eq:stable_solver}
(W+\Delta W)\widehat{\bm{z}}=\bm{y}+\Delta \bm{y},\quad\text{where}\quad \Vert\Delta W\Vert_\infty\leq c_1\ups,\quad \Vert\Delta \bm{y}\Vert_\infty \leq c_2\ups.
\end{align}
Here we focus on rounding errors, but \eqref{eq:stable_solver} and Theorem \ref{th:matrix_rational_function_bound} can easily be extended to account for other sources of error arising, e.g.~when an iterative method is used. We can now prove the following theorem:
\begin{theorem}
	\label{th:eps_expl_impl}
	Let Assumptions (A1), (A2) and (A3) and the assumptions of Theorem \ref{th:matrix_rational_function_bound} hold. Then, there exist constants $\bar{C}_1,\bar{C}_2,\bar{C}_3>0$ depending on the rounding function used such that, for all $n$,
	\begin{align}
	\Vert\mathlarger{\mathlarger{{\varepsilon}}}^n\Vert_\infty &\leq (\bar{C}_1(s,d,\lambda,K)M + \Delta t\bar{C}_2(s,d,\lambda,K)M_f)\ups \equiv \mathlarger{\mathlarger{{\varepsilon}}},\\
	\Vert\Theta^n\Vert_\infty &\leq \bar{C}_3(p,s,d,\lambda,K)\Delta t^p\ups \equiv \theta,\ \hspace{18pt}\Vert\eps^n\Vert_\infty \leq M\ups\equiv\varepsilon.
	\end{align}
	where $\lambda=\Delta t/h^2$ and the first two bounds only hold for $\ups$ sufficiently small. If the RK method is explicit $\bar{C}_1,\bar{C}_2$, and $\bar{C}_3$ still depend on $\lambda$, but not directly on $K$  (recall that $K=1/h$).
\end{theorem}

\begin{proof}
The bound on $\Vert\eps^n\Vert_\infty$ is a straightforward application of Assumption (A2). We first derive a bound for $\Vert\EPS^n\Vert_\infty$ and then we bound $\Vert\Theta^n\Vert_\infty$. Let $\tilde{\bm{U}}^{n+1}$ exactly satisfy \eqref{eq:num_scheme_finite_precision} with $\EPS^n$ set to zero:
\begin{align}
\label{eq:lemmaproof1}
\tilde{\bm{U}}^{n+1} = S(-\Delta tA)\widehat{\bm{U}}^{n} + \Delta t\bm{F}^n,
\end{align}
i.e.~$\tilde{\bm{U}}^{n+1}$ is the result of performing one timestep in exact arithmetic starting from $\widehat{\bm{U}}^{n}$.
We can then estimate $\EPS$ by computing an upper bound for the quantity
\begin{align}
\label{eq:lemmaproof2}
\Vert\mathlarger{\mathlarger{{\varepsilon}}}^n\Vert_\infty=\Vert\widehat{\bm{U}}^{n+1} - \tilde{\bm{U}}^{n+1}\Vert_\infty.
\end{align}

Let $\widehat{\bm{f}}_j^n = \bm{f}_j^n + (\Delta \bm{f}_j)^n$ for $j=1,\dots,s$ be the computed forcing terms due to finite precision and note that by Assumption (A1) we have $\Vert(\Delta \bm{f}_j)^n\Vert_\infty \leq C_fM_f\ups$. Let $\bm{z}_0=S(-\Delta t A)\widehat{\bm{U}}^{n}$, $\bm{y}=\tilde{S}(-\Delta t A)(-\Delta tA\widehat{\bm{U}}^{n})$ and $\bm{z}_j=\Delta tQ_j(-\Delta t A)\bm{f}_j^{n}$, and let $\widehat{\bm{z}}_j$ for $j=0,\dots,s$ and $\widehat{\bm{y}}$ be the result of computing the $\bm{z}_j$ and $\bm{y}$ in finite precision. Owing to Theorem \ref{th:matrix_rational_function_bound} and Assumptions (A1)-(A3), we obtain the upper bounds
\begin{align}
\label{eq:lemma_fakebound}
\begin{array}{ccc}
\Vert\widehat{\bm{z}}_0 - \bm{z}_0\Vert_\infty \leq C_1\ups M,&\Vert\widehat{\bm{z}}_j - \bm{z}_j\Vert_\infty \leq C_2 \Delta t\ups M_f, & \Vert\widehat{\bm{y}} - \bm{y}\Vert_\infty\leq C_3\Delta t^p\ups,\\
\hspace{11pt} \Vert\bm{z}_0\Vert_\infty \leq c_1M, & \hspace{11pt}\Vert\bm{z}_j\Vert_\infty \leq c_2\Delta t M_f, & \hspace{9pt} \Vert\bm{y}\Vert_\infty \leq c_3\Delta t^p,
\end{array}
\end{align}
where $c_1$, $c_2$, $c_3$, $C_1$, $C_2$ and $C_3$ are positive constants that depend on $s$, $d$ and $\lambda$ and might depend on $K$ for implicit methods. We first collect \eqref{eq:lemma_fakebound} into a bound for \eqref{eq:lemmaproof2}. We can write $\widehat{\bm{U}}^{n+1}$ as
\begin{align}
\widehat{\bm{U}}^{n+1} = \sum\limits_{j=0}^s\widehat{\bm{z}}_j\circ(1+\theta_j),\quad\text{where}\quad|\theta_j|\leq \gamma_{s}\bm{1}.
\end{align}
Here $\circ$ denotes the entrywise product. The vectors $\theta_j$ represent the error introduced by summing the $\widehat{\bm{z}}_j$ together in finite precision, cf.~\cite[Chapter 3]{higham2002accuracy}. Owing to \eqref{eq:lemma_fakebound} and Theorem \ref{th:matrix_rational_function_bound}, we have
\begin{align}
\label{eq:lemmaproof3}
\Vert\mathlarger{\mathlarger{{\varepsilon}}}^n\Vert_\infty &= \Vert\widehat{\bm{U}}^{n+1} - \tilde{\bm{U}}^{n+1}\Vert_\infty = \left|\left|\sum\limits_{j=0}^s\widehat{\bm{z}}_j\circ(1+\theta_j) - \sum\limits_{j=0}^s{\bm{z}}_j\right|\right|_\infty
\\
&\leq (1+\gamma_{s})\sum\limits_{j=0}^s \Vert\widehat{\bm{z}}_j - \bm{z}_j\Vert_\infty + \gamma_{s}\sum\limits_{j=0}^s\Vert\bm{z}_j\Vert_\infty\notag
\\
&\leq \bar{C}_1(s,d,\lambda,K) M\ups + \bar{C}_2(s,d,\lambda,K)\Delta t M_f \ups = \mathlarger{\mathlarger{{\varepsilon}}}.
\notag
\end{align}
This is the desired result. We omit the derivation of a close-form expression for $\bar{C}_1$ and $\bar{C}_2$.
%
The reasoning for $\Vert\Theta^n\Vert_\infty$ is similar. First we use \eqref{eq:lemma_fakebound} to bound
\begin{align}
\Vert\Delta \widehat{\bm{U}}\Vert_\infty \leq \Vert\bm{y}\Vert_\infty + \Delta t \sum\limits_{j=1}^s\Vert\bm{z}_j\Vert_\infty \leq c_3\Delta t^p+c_2\Delta tM_f.
\label{eq:_th_aux_1}
\end{align}
Following the same procedure as the one used to obtain \eqref{eq:lemmaproof3}, we get for $\Vert\tilde{\Theta}^n\Vert_\infty$ (cf.~\eqref{eq:Theta}),
\begin{align}
\Vert\tilde{\Theta}^n\Vert_\infty &\leq (1+\gamma_{s})\Vert\widehat{\bm{y}} - \bm{y}\Vert_\infty + (1+\gamma_{s})\sum\limits_{j=1}^s \Vert\widehat{\bm{z}}_j - \bm{z}_j\Vert_\infty + \gamma_{s}\left(\Vert\bm{y}\Vert_\infty + \sum\limits_{j=1}^s\Vert\bm{z}_j\Vert_\infty \right)
\notag\\
&\leq \tilde{C}_3(s,d,\lambda,K)\Delta t^p\ups + \bar{C}_2(s,d,\lambda,K) \Delta t M_f\ups,
\label{eq:_th_aux_2}
\end{align}
for some $\tilde{C}_3>0$ that does not depend on $K$ for explicit methods. The definition of $\Theta^n$, given entrywise in \eqref{eq:Theta}, implies that
\begin{align}
\label{eq:_th_aux_3}
\Vert\Theta^n\Vert_\infty\leq\Vert\tilde{\Theta}^n\Vert_\infty + \ups\Vert\Delta \widehat{\bm{U}}\Vert_\infty + \ups\Vert\tilde{\Theta}^n\Vert_\infty\Vert\Delta \widehat{\bm{U}}\Vert_\infty.
\end{align}
Pulling together equations \eqref{eq:_th_aux_1}, \eqref{eq:_th_aux_2}, and \eqref{eq:_th_aux_3} yields an upper bound for $\Vert\Theta^n\Vert_\infty$ and the desired result. We omit the derivation of a close-form expression for $\bar{C}_3$.
\end{proof}

\begin{remark}
	\label{rem:conditioning_lambda_blowup}
	The constants $\bar{C}_1$, $\bar{C}_2$ and $\bar{C}_3$ in Theorem \ref{th:eps_expl_impl} typically grow at least linearly with the condition number of $W(-\Delta t A)$, which is typically $O(\lambda)$. While it is well-known that $\lambda$ is bounded above for explicit methods due to stability reasons, this suggests that $\lambda$ must be kept small for implicit methods as well to avoid losing all accuracy.
\end{remark}


\subsection{Round-to-nearest causes stagnation}
\label{subsec:RtNstagnation}

As discussed in \cite[]{GuptaEtAl2015,Mikaitis2020,FasiMikaitis2020,ConnollyHighamMary2020} RtN computations are prone to stagnation. Explaining why this phenomenon occurs with RtN in an ODE/PDE context for small enough $\Delta t$ is simple. Let $x$ and $\epsilon$ be two exactly representable numbers. Then whenever the magnitude of $\epsilon$ is smaller than half the gap between $x$ and its representable neighbours, $x+\epsilon$ is rounded back to $x$. In formulas (cf.~Lemma 2.1 in \cite{higham2002accuracy}),
\begin{align}
\fl(x+\epsilon) = x,\quad\text{whenever}\quad \frac{u}{2}|x|\geq |\epsilon|.
\end{align}

We now consider stagnation in our numerical scheme. For simplicity, we only consider a delta form implementation of the RK method. Let $\uu(t,\bm{x})$ be the solution of \eqref{eq:main_eqn}, then stagnation will approximately occur when
\begin{align}
\label{eq:stagn_cond}
\frac{u}{2}|\uu(t_n,\bm{x}_{\bm{i}})| \approx \frac{u}{2}|\widehat{\bm{U}}^n_{\bm{i}}| \geq |\Delta \widehat{\bm{U}}^n_{\bm{i}}| = |\widehat{\bm{U}}^{n+1}_{\bm{i}}-\widehat{\bm{U}}^n_{\bm{i}}| \approx \Delta t |\dot{\uu}(t_n,\bm{x}_{\bm{i}})|.
\end{align}
This shows that $\widehat{\bm{U}}^n_{\bm{i}}$ will not be updated whenever
\begin{align}
\label{eq:rough_stagn_cond}
|\uu(t_n,\bm{x}_{\bm{i}})| \gtrapprox 2(\Delta t/u)|\dot{\uu}(t_n,\bm{x}_{\bm{i}})|.
\end{align}
This gives a condition which is clearly discretization-dependent and shows that stagnation will lead to different discrete solutions for different values of $\Delta t$ (hence $h$) and the unit roundoff $u$. According to \eqref{eq:rough_stagn_cond}, when a non-zero steady-state solution exists stagnation always occurs, since $|\dot{\uu}(t,\bm{x})|\rightarrow 0$ as $t\rightarrow\infty$. Furthermore, if $\uu(0,x),\dot{\uu}(0,x)\neq 0$ the discrete solution will always stagnate at the initial condition for $\Delta t$ small enough.

Overall, this simple argument already shows that the steady-state solution obtained with RtN is discretization, initial condition and precision dependent. This argument can be made rigorous, but we believe that the fact that stagnation occurs in this setting is both obvious and well known in the community, so we omit further details. We remark that not using the delta form only worsened the stagnation behavior in numerical experimentations.

\section{Global rounding error analysis}
\label{sec:global}

We now derive expressions for the global rounding errors in terms of $\EPS$, $\theta$, and $\varepsilon$. From Theorem \ref{th:eps_expl_impl} we know that the worst-case error bounds $|\EPS^n_{\bm{i}}|\leq \EPS$, $|\eps^n_{\bm{i}}|\leq \varepsilon$, and $|\Theta^n_{\bm{i}}|\leq \theta$ hold. However, more precise bounds can be obtained for SR owing to the following lemma:
\begin{lemma}
	\label{lemma:epsSR_properties}
	With SR we have that $\eps^n_{\bm{i}}$ is independent from $\eps^n_{\bm{j}}$ for all $n,{\bm{i}}\neq {\bm{j}}$. Furthermore, $\eps^n_{\bm{i}}$ satisfies $\E[\eps^n_{\bm{i}}] = \E\left[\eps^n_{\bm{i}}\left|\left\{\{\eps^m_{\bm{j}}\}_{{\bm{j}}=\bm{1}}^{{\bm{j}}=(K-1)\bm{1}}\right\}_{m=1}^{n-1}\right.\right]=0$ for all $n,{\bm{i}}$.
\end{lemma}

\begin{proof}
	The thesis is a consequence of the definition of $\eps^n$ and of the statistical properties of the SR roundoff errors, cf.~Lemma \ref{lemma:sr_properties}.
\end{proof}

\noindent We can now distinguish two different scenarios:
\begin{enumerate}
	\item[1)] \textbf{Worst-case (RtN):} we can only assume that $|\eps^n_i|\leq \varepsilon$ for all $n,i$.
	\item[2)] \textbf{Stochastic rounding:} Lemma \ref{lemma:epsSR_properties} holds. 
\end{enumerate}
\begin{remark}
	\label{rem:full_error}
	From now on we just investigate the global error deriving from $\eps^n$ for simplicity. However, the error bounds obtained also apply to the global errors deriving from the other error terms after making the following substitutions:
	\begin{itemize}
		\item Scheme not in delta form (both RtN and SR): simply replace $\varepsilon$ with $\EPS$ in the results from Scenario 1).
		\item Scheme in delta form, RtN: the total global error is obtained by replacing $\varepsilon$ with $\varepsilon+\theta$ in the results from Scenario 1).
		\item Scheme in delta form, SR: the total global error consists of two additive terms. The former is given by Scenario 2) and the latter is given by replacing $\varepsilon$ with $\theta$ in Scenario 1). While our theory does not show that the $\theta$ error term is negligible for small $\Delta t$, the effects of this second error term do not actually appear in numerical experimentation (cf.~Section \ref{sec:num_res}).
	\end{itemize}
\end{remark}
\begin{remark}
	\label{rem:exact_IC}
	In this section we assume for simplicity that the initial condition can be represented exactly so that $\widehat{\bm{U}}^0 = \bm{U}^0$ and $\EPS^{-1}=\eps^{-1}=\bm{0}$.
\end{remark}

We present a summary of the global error asymptotic rates derived in this section for fixed $\lambda$ in Table \ref{tab:rates}. We use both the infinity norm and the discrete $L^2$ norm, defined as $\Vert\cdot\Vert_{L^2}=K^{-d/2}\Vert\cdot\Vert_2$. The Scenario 1) infinity norm results for $d>1$ depend on the assumption that the matrix $S(-\Delta t A)$ only has non-negative entries. This requirement also implies that a discrete maximum principle holds and is automatically satisfied for Forward Euler and the popular RK4 method in their stability region, for Backward Euler for all $\lambda$, and for the Crank-Nicolson method for $\lambda\leq 1/d$.

\begin{table}[h!]
	\centering
	\begin{tabular}{lclll}
		\toprule
		\multicolumn{1}{c}{Scenario} & \multicolumn{1}{c}{Norm} & \multicolumn{1}{c}{1D}          & \multicolumn{1}{c}{2D}                                & \multicolumn{1}{c}{3D}          \\ \midrule
		1) RtN      &     $\infty$                 & $O(\varepsilon\Delta t^{-1})$   & $O(\varepsilon\Delta t^{-1}\ell(\Delta t)^s)$         & $O(\varepsilon\Delta t^{-1-s/2})$ \\
		1) RtN     & $L^2$                          & $O(\varepsilon\Delta t^{-1})$   & $O(\varepsilon\Delta t^{-1})$                      & $O(\varepsilon\Delta t^{-1})$ \\
		2) SR            & $\infty$                   & $O(\varepsilon\Delta t^{-1/4}\ell(\Delta t)^{1/2})$ & $O(\varepsilon \ell(\Delta t))$                 & $O(\varepsilon \ell(\Delta t)^{1/2})$                \\ 
		2) SR         & $L^2$                      & $O(\varepsilon\Delta t^{-1/4})$ & $O(\varepsilon \ell(\Delta t)^{1/2})$ & $O(\varepsilon)$ \\ \bottomrule
	\end{tabular}
	\caption{Asymptotic global rounding error blow-up rates for fixed $\lambda$ in the infinity and discrete $L^2$ norm. Here $\ell(\Delta t)=|\log(\lambda^{-1}\Delta t)|$. The error measures used for the worst-case scenario are either  $\Vert\bm{E}^n\Vert_\infty$ or $\Vert\bm{E}^n\Vert_{L^2}$. The value of $s$ is $0$ if $S(-\Delta t A)$ only has non-negative entries and $s=1$ otherwise. For the SR case the error measures are either $\E[\Vert\bm{E}^n\Vert^2_\infty]^{1/2}$ or $\E[\Vert\bm{E}^n\Vert_{L^2}^2]^{1/2}$.
	}
	\label{tab:rates}
	\vspace{0pt}
\end{table}

\begin{remark}
	For ease of notation, in this section we write e.g.~$\sum_{\bm{k}}$ in place of $\sum_{\bm{k}=\bm{1}}^{\bm{k}=(K-1)\bm{1}}$ to indicate a sum over all indices of the spatial degrees of freedom.
\end{remark}
Some preliminary analysis is needed before we can go over each scenario. We start by subtracting the exact delta form scheme \eqref{eq:num_scheme} from the approximate \eqref{eq:num_scheme_finite_precision_delta} and we set $\Theta^n=0$ to get a recursive relation for the global rounding error $\bm{E}^n=\widehat{\bm{U}}^n-\bm{U}^n$ deriving from $\eps^n$:
\begin{align}
\label{eq:num_scheme_global_error}
\bm{E}^{n+1} = S(-\Delta tA)\bm{E}^{n} + \eps^n.
\end{align}
The first step is to expand $\bm{E}^n$ in terms of the eigenvectors of $A$ (cf.~\eqref{eq:eigenpairsA}). We write $\bm{E}^n = \sum_{\bm{k}}a_{\bm{k}}^n\bm{v}_{\bm{k}}^h$ and we plug this into the above to obtain
\begin{align}
\sum_{\bm{k}}a_{\bm{k}}^{n+1}\bm{v}_{\bm{k}}^h=\sum_{\bm{k}}a_{\bm{k}}^nS(s_{\bm{k}})\bm{v}_{\bm{k}}^h + \eps^n,
\end{align}
where $s_{\bm{k}}=-\Delta t \lambda_{\bm{k}}^h$. Multiplying both sides by $(\bm{v}_{\bar{\bm{k}}}^h)^T$ we get, due to orthogonality and \eqref{eq:eigenpairsA_bounds},
\begin{align}
a_{\bar{\bm{k}}}^{n+1}=a_{\bar{\bm{k}}}^nS(s_{\bm{\bar{k}}}) + \left(\frac{2}{K}\right)^d(\eps^n,\bm{v}_{\bar{\bm{k}}}^h),
\end{align}
where we write $(\bm{a},\bm{b})=\bm{a}^T\bm{b}$ for any two vectors $\bm{a}$, $\bm{b}$. We now drop the bar in $\bar{\bm{k}}$. Since $\bm{E}^0=\bm{0}$ by assumption (cf.~Remark \ref{rem:exact_IC}), we have $a_{\bm{k}}^0=0$ for all $\bm{k}$ and we obtain
\begin{align}
a^n_{\bm{k}}=\left(\frac{2}{K}\right)^d\ \sum\limits_{m=0}^{n-1}S(s_{\bm{k}})^{n-1-m}(\eps^m,\bm{v}_{\bm{k}}^h).
\end{align}
Before working through the various correlation assumptions it is useful to prove the following theorem, as it plays an important role in the whole section.
\begin{theorem}
	\label{thm:Sseriesbounds}
	Let $\mathcal{A}=\{z\leq 0 : |S(z)| < 1\}$, $\Lambda = (-4d \lambda,0)$, and assume that $\lambda$ is chosen such that $\left(\Lambda \cup \{-4 d \lambda\}\right) \subset \mathcal{A}$. Further assume that $S(z)$ has no poles on the negative real axis. Then, there exist constants $C(d,\lambda)>0,\bar{C}(d,\lambda),\tilde{C}(d,\lambda)>0$ such that
	\begin{align}
	\label{eq:Sseriesbounds1}
	 \sum_{\bm{k}}\dfrac{1}{1-S(s_{\bm{k}})^2} \leq \sum_{\bm{k}}\dfrac{1}{1-|S(s_{\bm{k}})|}&\leq
	 \phi_d(\Delta t)=\left\{
	 \begin{array}{lr}
	 C(1,\lambda)\Delta t^{-1}, & d=1,\\
	 C(2,\lambda)\Delta t^{-1}|\log(\lambda^{-1}\Delta t)|, & d=2,\\
	 C(3,\lambda)\Delta t^{-3/2}, & d=3,\\
	 \end{array}
	 \right.\\
	 \sum_{\bm{k}}\dfrac{1}{(1-|S(s_{\bm{k}})|)^2} &\leq \bar{C}(d,\lambda)^2\Delta t^{-2},
	 \label{eq:Sseriesbounds2}
	 \\
	 K^{-d}\sum_{\bm{k}}\dfrac{(\bm{1},\bm{v}^h_{\bm{k}})}{1-|S(s_{\bm{k}})|} &\leq \tilde{C}(d,\lambda)\Delta t^{-1}.
	 \label{eq:Sseriesbounds3}
	\end{align}
\end{theorem}

\begin{proof}
	See \ref{sec:appB}.
\end{proof}

\noindent
We now derive estimates for the global error for both scenarios.

\subsection{Worst-case error (round-to-nearest)}
In this case, the only thing we know is that $|\eps^n_i|\leq \varepsilon$ for all $n,i$. We have
\begin{align*}
|a^n_{\bm{k}}|\leq 2^d\varepsilon\left(\frac{K-1}{K}\right)^d\ \sum\limits_{m=0}^{n-1}|S(s_{\bm{k}})|^{n-1-m}\leq 2^d\dfrac{\varepsilon}{1-|S(s_{\bm{k}})|},
\end{align*}
since for stability $|S(s_{\bm{k}})| < 1$. Let $\bm{a}^n\in\R^{(K-1)^d}$ be such that $(\bm{a}^n)_{\bm{k}}=a^n_{\bm{k}}$. We then have that (cf.~\eqref{eq:eigenpairsA_bounds} and Theorem \ref{thm:Sseriesbounds})
\begin{align*}
\Vert\bm{E}^n\Vert_\infty&\leq\sum_{\bm{k}}|a_{\bm{k}}^n|\leq 2^d\varepsilon\ \sum_{\bm{k}}\dfrac{1}{1-|S(s_{\bm{k}})|}\leq 2^d \varepsilon\phi_d(\Delta t),
\\
\Vert\bm{E}^n\Vert_{L^2}^2&=K^{-d}\left(\frac{K}{2}\right)^d\ \Vert\bm{a}^n\Vert_2^2\leq \varepsilon^22^d\ \sum_{\bm{k}}\dfrac{1}{(1-|S(s_{\bm{k}})|)^2}\leq 2^d\bar{C}(d,\lambda)^2\varepsilon^2\Delta t^{-2},\\
\Rightarrow \Vert\bm{E}^n\Vert_{L^2}&\leq 2^{d/2}\bar{C}(d,\lambda)\varepsilon\Delta t^{-1}.
\end{align*}
Here we used the equality $\Vert\bm{E}^n\Vert_2^2=(K/2)^d \Vert\bm{a}^n\Vert_2^2$ which comes from the fact that the $\bm{v}^h_{\bm{k}}$ are orthogonal. The function $\phi_d(\Delta t)$ was defined in \eqref{eq:Sseriesbounds1}: $\phi_1(\Delta t)=C(1,\lambda)\Delta t^{-1}$, $\phi_2(\Delta t)=C(2,\lambda)\Delta t^{-1}|\log(\lambda^{-1}\Delta t)|$, $\phi_3(\Delta t)=C(3,\lambda)\Delta t^{-3/2}$.

The infinity norm estimate in the 2D and 3D case can be refined when the matrix $S(-\Delta t A)$ has only non-negative entries. In this case, a discrete maximum (minimum) principle holds and for an arbitrary sequence of non-negative (respectively non-positive) vectors $\bm{f}^0,\dots,\bm{f}^n$ we have that the vector sequence $\bm{W}^n$ satisfying
\begin{align}
\label{eq:tempW}
\bm{W}^{n+1} = S(-\Delta t A)\bm{W}^{n} + \bm{f}^n=\sum\limits_{m=0}^{n}S(-\Delta t A)^m\bm{f}^m,
\end{align}
with $\bm{W}^0=\bm{0}$ is also non-negative (respectively non-positive). Now, let $\bm{M}^n$ be the solution of \eqref{eq:num_scheme_global_error} in which $\eps^n$ is replaced by $\varepsilon\bm{1}$. We then have that $\bm{E}^n-\bm{M}^n$ and $\bm{M}^n + \bm{E}^n$ satisfy equation \eqref{eq:tempW} with right hand sides that are non-positive and non-negative respectively, giving $\bm{E}^n-\bm{M}^n\leq 0$ and $\bm{M}^n + \bm{E}^n\geq 0$. This gives $|\bm{E}^n|\leq \bm{M}^n$ and
\begin{align*}
|\bm{E}^n_{\bm{j}}| &\leq \bm{M}^n_{\bm{j}} = \varepsilon\left(\frac{2}{K}\right)^d\sum_{\bm{k}}\sum_{m=0}^{n-1} S(s_{\bm{k}})^{n-1-m}(\bm{1},\bm{v}^h_{\bm{k}})\bm{v}^h_{\bm{k},{\bm{j}}}\leq \varepsilon 2^dK^{-d}\sum_{\bm{k}}\frac{(\bm{1},\bm{v}^h_{\bm{k}})}{1-|S(s_{\bm{k}})|}\\
&\leq \varepsilon 2^d\tilde{C}(d,\lambda)\Delta t^{-1},
\end{align*}
where we used \eqref{eq:eigenpairsA_bounds} in the second step and Theorem \ref{thm:Sseriesbounds} in the third. Since the bound is independent from $j$ this is also a valid bound for $\Vert\bm{E}^n\Vert_{\infty}$.

We remark that the fact that RtN rounding errors grow like $O(u\Delta t^{-1})$ is well-known in an ODE context \cite[]{Henrici1962,Henrici1963}. This error growth rate was also more recently observed by \cite{Jezequel1995}, where the author studies the propagation of rounding errors in the solution of the heat equation via the forward Euler method with RtN and towards-zero rounding.


\subsection{Stochastic rounding error}
Here $\E[\eps^n_{\bm{i}}]=0$ for all $n,{\bm{i}}$, and $\E[\eps^n_{\bm{i}}\eps^k_{\bm{j}}]=0$ unless $n=k$ and ${\bm{i}}={\bm{j}}$ due to mean-independence, cf.~Lemma \ref{lemma:epsSR_properties}. We therefore have that owing to \eqref{eq:eigenpairsA_bounds},
\begin{align*}
\E[(\eps^m,\bm{v}_{\bm{k}}^h)^2]=\sum\limits_{{\bm{i}},{\bm{j}}}\E[\eps^m_{\bm{i}}\eps^{m}_{\bm{j}}](\bm{v}_{\bm{k}}^h)_{\bm{i}}(\bm{v}_{\bm{k}}^h)_{\bm{j}} = \eps^2\Vert\bm{v}_{\bm{k}}\Vert_2^2=\eps^2\left(\frac{K}{2}\right)^d.
\end{align*}
This yields,
\begin{align*}
\E[\Vert\bm{E}^n\Vert_2^2]&=\left(\frac{K}{2}\right)^d \E[\Vert\bm{a}^n\Vert_2^2]=\left(\frac{2}{K}\right)^{d}\sum\limits_{\bm{k}}\sum\limits_{m}^{n-1}S(s_{\bm{k}})^{2(n-1-m)}\E[(\eps^m,\bm{v}^h_{\bm{k}})^2] \notag \\
&\leq\varepsilon^2\sum\limits_{\bm{k}}\frac{1}{1-S(s_{\bm{k}})^2}\leq \varepsilon^2\phi_d(\Delta t),\\
\Rightarrow\quad \E[\Vert\bm{E}^n\Vert_{L^2}^2]^{1/2}&\leq \varepsilon K^{-d/2}\sqrt{\phi_d(\Delta t)}=\varepsilon \lambda^{-d/4}\Delta t^{d/4}\sqrt{\phi_d(\Delta t)}.
\end{align*}
We have again used Theorem \ref{thm:Sseriesbounds}. The bound is $O(\Delta t^{-1/4})$ in 1D, $O(|\log(\lambda^{-1}\Delta t)|^{1/2})$ in 2D and $O(1)$ in 3D for fixed $\lambda$.

Obtaining a bound for $\E[\Vert\bm{E}^n\Vert^2_\infty]^{1/2}$ is more difficult and requires the combined use of the concept of a martingale, Azuma-Hoeffding inequality and a maximal inequality for sub-Gaussian random variables. We start by defining martingales:
\begin{definition}[Martingale]
	A sequence of random variables $\{X_n\}_{n=0}^N$ is a martingale if $\E[|X_n|]<\infty$ and $\E[X_n|X_0,\dots,X_{n-1}]=X_{n-1}$.
\end{definition}
Note that this is not a complete definition of a martingale, but it is sufficient for our purposes; see Section 2.3 in \cite{KloedenPlaten2013} for a more complete definition. Azuma-Hoeffding inequality reads:
\begin{theorem}[Azuma-Hoeffding inequality (\cite{Hoeffding1963,Azuma1967})]
	\label{lemma:Azuma_ineq_standard}
	For a martingale $\{X_n\}_{n=0}^N$ with ${|X_n-X_{n-1}|\leq c_n}$ a.s.,
	\begin{align}
	\P(|X_N-X_0| > t) < 2\exp\left(-\dfrac{t^2}{2\sum_{n=1}^Nc_n^2}\right).
	\end{align}
\end{theorem}
A sub-Gaussian random variable is a random variable for which the tails of its distribution decay as fast as a Gaussian distribution. More formally,
\begin{definition}[Sub-Gaussian random variable]
	A zero-mean random variable $X(\omega)$ is sub-Gaussian with proxy variance $\sigma^2>0$ if
	\begin{align}
	\P(|X|>t) \leq 2\exp\left(-\frac{t^2}{2\sigma^2}\right).
	\end{align}
\end{definition}
Clearly, the quantity $X_N-X_0$ in Azuma-Hoeffding inequality is sub-Gaussian. Here we will be interested in controlling the maximum of many sub-Gaussian random variables. The following maximal inequality will be useful:
\begin{theorem}
	\label{th:maximal}
	Let $\{X_i\}_{i=1}^N$ be a collection of sub-Gaussian random variables with variance proxy $\sigma^2$. Then
	\begin{align}
	\E[\max_i|X_i|^2]\leq 2\sigma^2\log(2eN).
	\end{align}
	The random variables in the collection need not be independent.
\end{theorem}

\begin{proof}
	Set $Z=\max_i|X_i|^2$ and let $P(z)$ and $P^{-1}(z)$ be the cumulative distribution function and the quantile function of $Z$ respectively. Then,
	\begin{align}
	\label{eq:_proof_maximal_ineq1}
	\E[Z]=\int_0^1P^{-1}(s)ds=\int_0^1P^{-1}(1-s)ds.
	\end{align}
	Furthermore, since $\{X_i\}_{i=1}^N$ are sub-Gaussian, we have
	\begin{align}
	1-P(t)=\P(Z>t)\leq \sum_{i=1}^N\P(|X_i|>t) \leq 2N\exp\left(-\frac{t}{2\sigma^2}\right).
	\end{align}
	Setting $1-P(t)=s$ so that $t = P^{-1}(1-s)$, we obtain $\exp\left(\frac{t}{2\sigma^2}\right)\leq \frac{2N}{s}$, which implies that $t \leq 2\sigma^2\log\left(\frac{2N}{s}\right)$, and hence from \eqref{eq:_proof_maximal_ineq1} we get
	\begin{align}
	\E[Z]\leq 2\sigma^2\int_0^1\log\left(\frac{2N}{s}\right)ds = 2\sigma^2\log(2eN),
	\end{align}
	which is the thesis.
\end{proof}

To obtain a bound for $\E[\Vert\bm{E}^n\Vert^2_\infty]^{1/2}$ we write each $\bm{E}^n_{\bm{i}}$ as
\begin{align}
\bm{E}^n_{\bm{i}} = \sum\limits_{m=0}^{n-1}\sum\limits_{{\bm{j}}} (S(-\Delta t A)^{n-1-m})_{{\bm{i}}{\bm{j}}}\eps^m_{\bm{j}} = \sum\limits_{m=0}^{n-1}\sum\limits_{{\bm{j}}}Z_{m,{\bm{j}}} = \sum\limits_{k=1}^{(n-1)(K-1)^d} Z_k,
\end{align}
where we just re-indexed the last sum with $k=\sum_{i=1}^d\bm{j}_i(K-1)^{i-1} + m(K-1)^d$. Now, if we set $P=(n-1)(K-1)^d$, $X_p = \sum_{k=1}^pZ_k$, with $X_0=0$, and $X_P=\bm{E}^n_{\bm{i}}$, we have that $\{X_p\}_{p=0}^P$ is a martingale: clearly $\E[|X_p|]<\infty$ and
\begin{align}
\E[X_p|X_0,\dots,X_{p-1}]=\E[Z_p|X_0,\dots,X_{p-1}] + X_{p-1}=X_{p-1},
\end{align}
due to the properties of $\eps^n$, cf.~Lemma \ref{lemma:epsSR_properties}. We can therefore apply Azuma-Hoeffding inequality with
\begin{align}
|X_p-X_{p-1}|\leq \varepsilon|S(-\Delta t A)^{n-m-1}|_{{\bm{i}}{\bm{j}}}=c_p,
\end{align}
for $p=\sum_{i=1}^d\bm{j}_i(K-1)^{i-1}+m(K-1)^d$. $\bm{E}^n_{\bm{i}}$ is therefore sub-Gaussian with proxy variance
\begin{align}
\label{eq:sigma2_SR_1}
\sigma^2 = \sum_pc_p^2 = \varepsilon^2 \sum\limits_{m=0}^{n-1}\sum\limits_{{\bm{j}}}\left((S(-\Delta t A)^{n-m-1})_{\bm{i}\bm{j}}\right)^2. 
\end{align}

The next step is to derive a nicer expression for $\sigma^2$. We define the auxiliary vector $\bar{\bm{E}}^n$, which solves equation \eqref{eq:num_scheme_global_error} in which $\eps^n$ has been replaced by $\bar{\eps}^n(\omega)$ with $\bar{\eps}^n_i=\pm\varepsilon$ with equal probability and with $\bar{\eps}^n_{\bm{i}}$ and $\bar{\eps}^m_{\bm{j}}$ independent unless ${\bm{i}}={\bm{j}}$ and $n=m$. The key step now is to realize that by construction we have $\sigma^2=\E[(\bar{\bm{E}}^n_{\bm{i}})^2]$, where the expectation can be bounded as follows.
\begin{align*}
\E[(\bar{\bm{E}}^n_{\bm{i}})^2]&=\varepsilon^2\left(\frac{2}{K}\right)^{2d}\ \sum\limits_{m=0}^{n-1}\sum\limits_{{\bm{j}}}\left(\sum\limits_{\bm{k}}S(s_{\bm{k}})^{n-1-m}(\bm{v}_{\bm{k}}^h)_{{\bm{i}}}(\bm{v}^h_{\bm{k}})_{{\bm{j}}}\right)^2\\
&= \varepsilon^2\left(\frac{2}{K}\right)^{2d}\ \sum\limits_{m=0}^{n-1}\sum\limits_{\bm{k},\tilde{\bm{k}}}S(s_{\bm{k}})^{n-1-m}S(s_{\tilde{\bm{k}}})^{n-1-m}(\bm{v}_{\bm{k}}^h)_{{\bm{i}}}(\bm{v}_{\tilde{\bm{k}}}^h)_{{\bm{i}}}\sum\limits_{\bm{j}}(\bm{v}^h_{\bm{k}})_{{\bm{j}}}(\bm{v}^h_{\tilde{\bm{k}}})_{{\bm{j}}}\\
&=\varepsilon^2\left(\frac{2}{K}\right)^{d}\ \sum\limits_{\bm{k}}\sum\limits_{m=0}^{n-1}S(s_{\bm{k}})^{2(n-1-m)}(\bm{v}_{\bm{k}}^h)_{{\bm{i}}}^2\leq \varepsilon^2\left(\frac{2}{K}\right)^{d}\ \sum\limits_{\bm{k}}\frac{1}{1-S(s_{\bm{k}})^2}.
\end{align*}
Here we used the fact that $\sum_{\bm{j}}(\bm{v}^h_{\bm{k}})_{{\bm{j}}}(\bm{v}^h_{\tilde{\bm{k}}})_{{\bm{j}}}=0$ if $\bm{k}\neq\tilde{\bm{k}}$ due to the orthogonality of the $\bm{v}_{\bm{k}}^h$ and the relations in \eqref{eq:eigenpairsA_bounds}. We can now apply Theorem \ref{thm:Sseriesbounds} to obtain
\begin{align}
\sigma = \E[(\bar{\bm{E}}^n_{\bm{i}})^2]^{1/2}\leq 2^{d/2}\varepsilon K^{-d/2}\sqrt{\phi_d(\Delta t)}=2^{d/2}\lambda^{-d/4}\varepsilon \Delta t^{d/4}\sqrt{\phi_d(\Delta t)}.
\end{align}
We conclude by invoking Theorem \ref{th:maximal} and the fact that $\bm{E}^n_i$ is sub-Gaussian with proxy variance $\sigma^2$, giving:
\begin{align}
\E[\Vert\bm{E}^n\Vert^2_\infty]^{1/2}\leq 2^{(d+1)/2}\lambda^{-d/4}\varepsilon \Delta t^{d/4}\sqrt{\phi_d(\Delta t)}\sqrt{d/2\log(\lambda^{-1}\Delta t) + \log(2e)}.
\end{align}
For fixed $\lambda$, the bound is $O(\Delta t^{-1/4}|\log(\lambda^{-1}\Delta t)|^{1/2})$ in 1D, $O(|\log(\lambda^{-1}\Delta t)|)$ in 2D and $O(|\log(\lambda^{-1}\Delta t)|^{1/2})$ in 3D.

\begin{remark}
	In this section we have shown that the use of SR leads to smaller error estimates. While these upper bounds\footnote{We remark that when working with rounding errors the only lower bound achievable is zero since floating-point computations can be exact in some specific scenarios.} do not necessarily establish that SR will be more accurate than RtN, the fact that deterministic rounding modes lead to a linear growth of the error is well known \cite[]{Henrici1962,Henrici1963,Jezequel1995,ConnollyHighamMary2020}. Indeed, we show in the next section that all our estimates are sharp, and that global errors behave as predicted by our bounds.
\end{remark}

\begin{remark}
	In this paper we did not discuss adaptive timestepping. We remark that this is not just for the sake of simplicity. Adaptivity in RK schemes typically relies on error estimates based on the comparison of two schemes that converge with different orders of accuracy. The problem when working in reduced precision is that these two schemes would not actually converge, but instead produce errors that, as we have shown in this section, stagnate or blow up as $\Delta t\rightarrow 0$. This behavior is likely to disrupt most adaptive schemes. The design of an adaptive strategy that is robust to rounding errors is currently an open question that we leave for future work.
\end{remark}

\section{Numerical results}
\label{sec:num_res}


In this section we present some numerical experiments that support and complete the theory presented in the paper. The code and data used to generate the figures and results of this section are publicly available on Zenodo\footnote{See \url{https://doi.org/10.5281/zenodo.6146386}.}. All low-precision computations are emulated in software and performed using bfloat16. The fp16 format yields qualitatively similar results which we do not show here. Emulating low-precision operations is computationally expensive, especially for 2D and 3D problems. To make the computations faster, we employ a custom-built C++ emulator, Libchopping\footnote{Available at \url{https://bitbucket.org/croci/libchopping/}.}, interfaced to Python, that combines vectorization, OpenMP and MPI for parallelization and uses the Intel Math Kernel Library for random number generation. The basic emulation algorithm was inspired by the low-precision emulator of \cite{HighamPranesh2019}. Other low-precision emulators supporting SR are also available (cf.~\cite{fousse2007mpfr,zhang2019qpytorch,fasi2021algorithms}), see Section 7(d) in \cite{croci2022stochastic} for an overview.


\paragraph{Test problems} Throughout the rest of the section, unless otherwise indicated, we fix $h$ to be a power of two, $\lambda=(1/2 - 2^{-4})/d$, $T=\lceil 1/\Delta t \rceil \Delta t$ and $g(\bm{x})= G\in\R$ for all $\bm{x}$. The choice of $\lambda$ was made so as to satisfy the stability requirement for forward Euler (FE). Furthermore, we choose time-independent forcing terms so that the exact steady-state solutions are:
\begin{gather*}
\uu_{1D}(\infty,x)      = (4\,x(1-x))^2 + G,\quad
\uu_{2D}(\infty,\bm{x}) = (16\,xy(1-x)(1-y))^2 + G,\\
\uu_{3D}(\infty,\bm{x}) = (64\,xyz(1-x)(1-y)(1-z))^2 + G.\notag
\end{gather*}
Note that all of these solutions are symmetric with respect to $x{=}\frac{1}{2},y{=}\frac{1}{2},z{=}\frac{1}{2}$, satisfy the boundary conditions, and attain their maximum value of $1{+}G$ at the center of the domain. To obtain the corresponding forcing terms it is sufficient to take the negative Laplacian of the exact solutions.

\begin{figure}[h!]
	\centering
	\begin{subfigure}[]{0.45\textwidth}
		\centering
		\includegraphics[width=\textwidth]{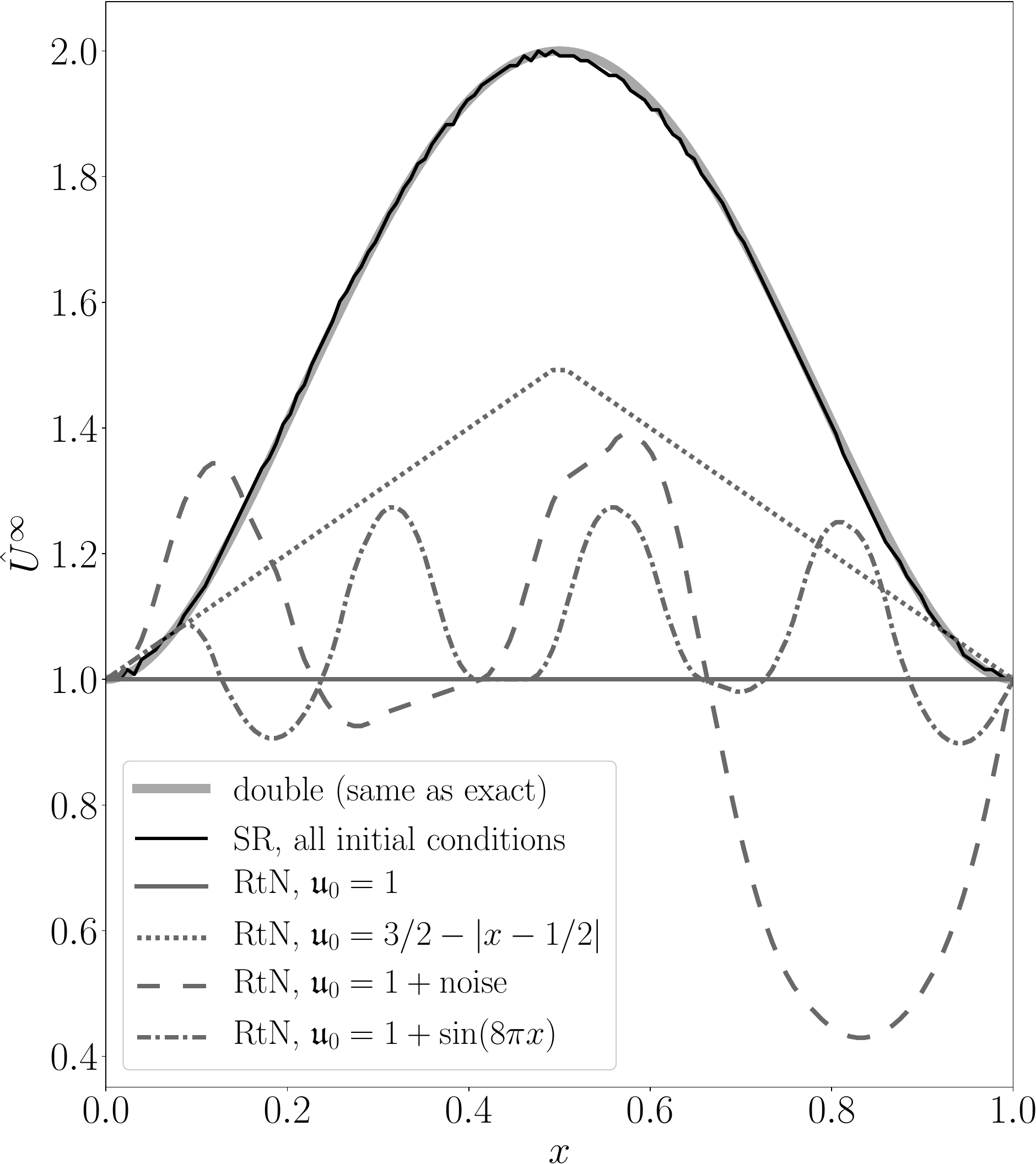}
	\end{subfigure}%
	~
	\begin{subfigure}[]{0.45\textwidth}
		\centering
		\includegraphics[width=\textwidth]{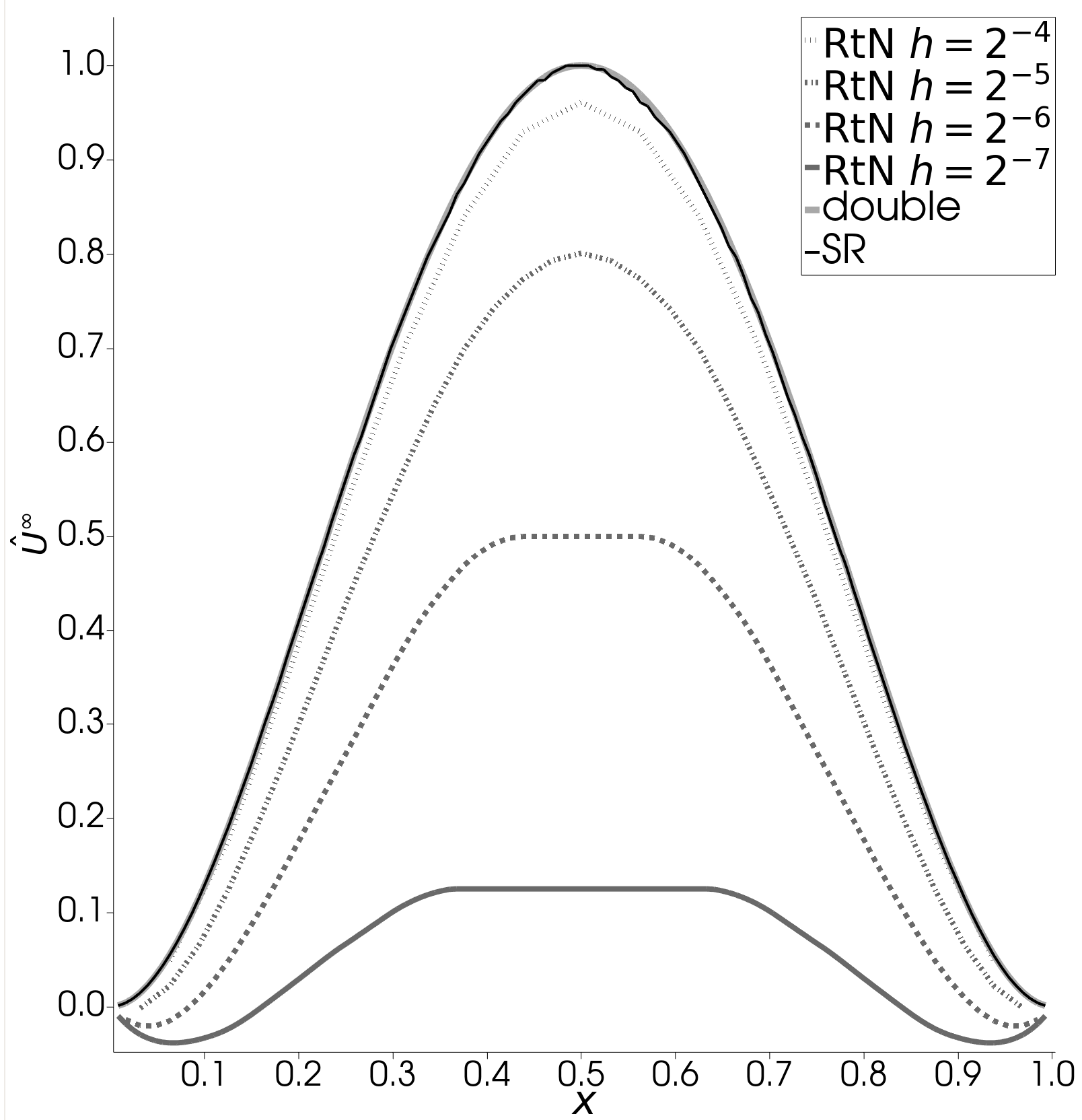}
	\end{subfigure}
	\caption{\textit{Comparison between FE and exact steady-state solutions in 1D (left) and in 2D (right, here the plot is over the $y=1/2$ line in the $xy$ plane) at $t=T=1$ for different initial conditions. In both plots we compare the exact steady-state solution (visibly the same as for double precision) with those obtained with bfloat16+RtN and bfloat16+SR. The SR solutions are always obtained with $h=2^{-7}$ and are as accurate as the precision allows.
			In the left figure, we fix $G=1$, $h=2^{-7}$ and we vary the initial condition leading to different steady-state RtN solutions. The $\text{noise}$ term in the initial condition has been obtained by sampling independent standard Gaussian random variables at each mesh node. We note that all SR solutions converge to the same steady-state modulo some barely visible oscillations. In the right figure, we fix $G=0$, $\uu_0=0$, and, for RtN only, we vary the mesh size. The RtN solutions converge to smaller steady-state solutions as $h$ is refined, significantly underestimating the true solution. On the finest grid the RtN graph even attains negative values.}}
	\label{fig:solplot}
	\vspace{0pt}
\end{figure}

In this section we mainly experiment with first order methods, namely FE in 1D, 2D and 3D and backward Euler (BE) in 1D and 2D. Experimentation with other popular Runge-Kutta methods, RK4, Crank-Nicolson, and first- and second-order RK-Chebyshev methods (RKC1 and RKC2; see \cite{VerwerHundsdorfer1990RKC}) yielded qualitatively similar results, which we mainly do not show here for the sake of brevity. The linear system solver used for BE was the Thomas algorithm in 1D and geometric multigrid (V-cycles) with Jacobi relaxation in 2D. Both solvers were implemented in (emulated) low precision. For the geometric multigrid solver, we observe monotone convergence of the residual and we declare convergence when there is no longer any significant reduction in the residual. Experimentally, this only happens when the relative residual is close to machine precision. 


Before proceeding, we show in Figure \ref{fig:solplot} what happens in practice when the heat equation is solved using double precision or bfloat16 with RtN and SR. The results in the left figure have been obtained with FE in 1D with $h=2^{-7}$ by setting $g(x)=G=1$ and by varying the initial condition. The results in the right figure have been obtained with FE in 2D by setting $g(x)=G=0$ and varying the mesh size: $h\in\{2^{-l}\}_{l=4}^{l=7}$. Both plots show the numerical and exact solutions at $t=T=1$.

The results are in agreement with the considerations of Section \ref{subsec:RtNstagnation}. Due to stagnation, the low-precision RtN solution stops being updated and the values obtained significantly underestimate the true solution. In Figure \ref{fig:solplot} (left), we note how different initial conditions lead to different steady-state solutions for RtN. Conversely, the SR computations always converge to the same steady state (whenever this is representable) or mildly oscillate around it at random (when it is not representable). Remarkably, RtN solutions stagnate very close to low-frequency initial conditions ($\uu_0=1$, $\uu_0=3/2-|x-1/2|$) and later for high-frequency ones ($\uu_0=1+\sin(8\pi x)$ and Gaussian noise). In fact, the latter led to oscillatory solutions for which the discrete Laplacian, and consequently the $\Delta \widehat{\bm{U}}^n$ terms, are larger in magnitude and counterbalance the $\Delta t$ in the stagnation condition \eqref{eq:stagn_cond}.

In Figure \ref{fig:solplot} (right) we observe that stagnation occurs even for zero initial conditions. In this case, however, stagnation occurs gradually as the mesh size (hence $\Delta t$) decreases, each time leading to different steady-state solutions. Furthermore, the steady-state solution seems to approach the zero initial condition as $\Delta t\rightarrow 0$. SR on the other hand is again impervious to stagnation and the bfloat16 results are as accurate as the working precision allows. In both plots we also observe that while RtN preserves symmetry due to its deterministic nature, SR does not. We note, however, that the asymmetry in the SR results is barely visible in the figure.

\noindent\textbf{Note:} From now on we set $\uu_0=G=1$ in all the numerical experiments.

\begin{remark}
	In our numerical experiments, the computation of the matrix-vector product $-A\widehat{\bm{U}}^n$ via two first-order differences (as described in Section \ref{subsec:matvecs}) always resulted to be exact in all dimensions. This behavior is likely a consequence of $h$ being a power of $2$. Conversely, using a direct second-order difference introduced rounding errors growing like $(h^{-2})$, as predicted by Statement 1 of Theorem \ref{th:2firstorderdiff2}. We do not show the latter results here for the sake of brevity.
\end{remark}

\paragraph{Local rounding errors}
In Figure \ref{fig:local2D_RtN_SR} we plot the worst-observed local rounding errors arising from the solution of the 2D test problem as a function of the timestep $\Delta t$. The 1D and 3D results, and the behavior of higher-order methods such as RK4 and RKC2 are qualitatively similar and we do not show them here. Recall that here $\lambda$ is fixed and is the same for both FE and BE (we investigate how $\lambda$ affects the BE results later). We consider three cases: 1) We use the delta form of the timestepping scheme and we use the more refined second-order differencing strategy (``delta form'' in the figure). 2) We use the delta form, but we directly use a second order difference (``naive matvec'' in the figure). 3) We do not use the delta form. We adopt the following error measures for the local rounding errors:
\begin{align}
\label{eq:local_rounding_error_measures}
\frac{1}{u}\max_n\frac{\Vert\text{err}^n\Vert_\infty}{\Vert\widehat{\bm{U}}^N\Vert_\infty},\quad\text{for RtN};\quad \frac{1}{u}\max_{n,\omega}\frac{\Vert\text{err}^n(\omega)\Vert_\infty}{\Vert\widehat{\bm{U}}^N(\omega)\Vert_\infty},\quad\text{for SR}.
\end{align}
Where $\text{err}^n$ is either $\Theta^n$ or $\EPS^n$ depending on whether the delta form is used or not. For SR the $\max$ is taken over all observed realizations across the number of Monte Carlo samples used to estimate the global errors (see later). We divide the errors by $\Vert\widehat{\bm{U}}^N\Vert_\infty$ to normalize the results: for this test problem the RtN solution has smaller magnitude due to stagnation, cf.~Figure \ref{fig:solplot}. We do not track the $\eps^n$ term here since we know from Theorem \ref{th:eps_expl_impl} that it is bounded above by $M\ups$ where $M=\max_n\Vert\widehat{\bm{U}}^n\Vert_\infty$ giving a maximum value of $1$ for RtN and of roughly $2$ for SR for the error measures in \eqref{eq:local_rounding_error_measures}. In fact, for this problem we observe that $\Vert\widehat{\bm{U}}^n\Vert_\infty\leq\Vert\widehat{\bm{U}}^N\Vert_\infty$ for all $n$ for RtN and that the same relation still approximately holds for SR.

\begin{figure}[h!]
	\centering
	\includegraphics[width=\linewidth]{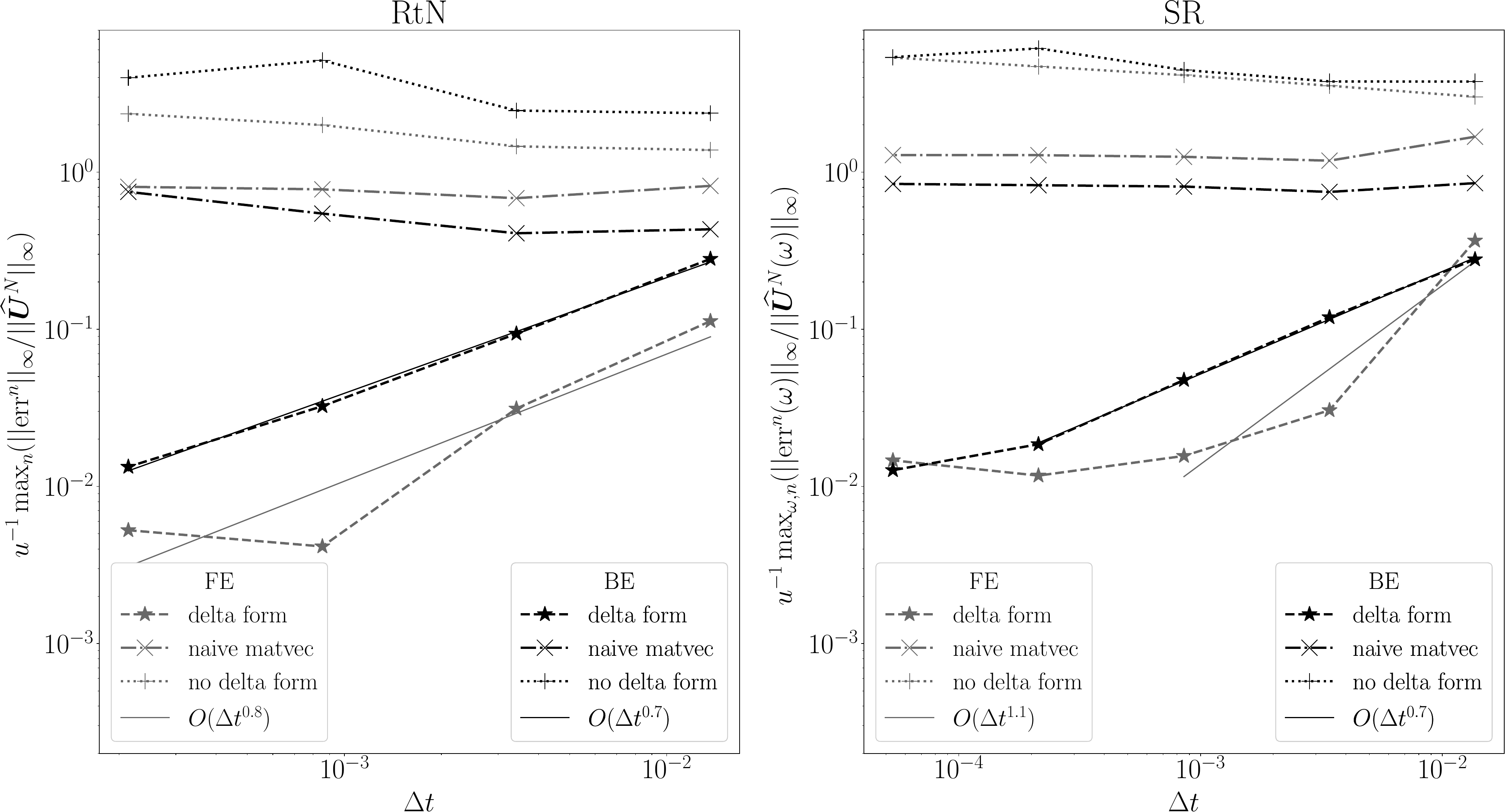}
	\caption{\textit{Worst-observed local errors as a function of $\Delta t$ for RtN (left) and SR (right) and for FE (black) and BE (grey) in 2D. The slopes of the solid lines have been estimated from the observed values via linear regression. The results match with the theory showing that the delta form and the refined differencing strategy are significantly more accurate. Refining $\Delta t$ further than the left figure values causes RtN to stagnate at $\bm{U}^0=\bm{1}$.}}
	\label{fig:local2D_RtN_SR}
	\vspace{-12pt}
\end{figure}

From Figure \ref{fig:local2D_RtN_SR} it is clear that using the delta form with the refined differencing strategy yields the minimum error. In this case Assumption (A3) appears to be satisfied since the error decreases with $\Delta t$ at a rate $p\in[1/2,1]$. For SR this rate eventually plateaus to $p=0$. To understand better this behavior, we look at FE for simplicity and note that since the matrix-vector products are here computed exactly, we can write the (rounded) delta form update as
\begin{align}
\Delta \widehat{\bm{U}}^n + \tilde{\Theta}^n = (\Delta t A\widehat{\bm{U}}^n + \Delta t (\bm{f}^n + \Delta \bm{f}^n))(1+\bm{\delta}).
\end{align}
Here $\Delta \bm{f}^n$ satisfies $\Vert\Delta \bm{f}^n\Vert_{\infty}\leq C_fM_f\ups$ (Assumption (A1) holds), $\Vert\bm{\delta}\Vert_\infty < \ups$ and we have ignored the error due to multiplication by $\Delta t$ for simplicity. For small $\Delta t$, the $\Delta t \Delta \bm{f}^n$ error term is negligible and the worst-case error in the update $\Theta^n$ (cf.~equation \eqref{eq:Theta}) is, to leading order,
\begin{align}
\Theta^n = O(\ups\Delta t | A\widehat{\bm{U}}^n|).
\end{align} 
Let $m\in\{0,1,2\}$. If $\widehat{\bm{U}}^n$ were the discrete version of an $m$-times differentiable function of space, we would expect $\Delta t|A\widehat{\bm{U}}^n|=O(\Delta t h^{m-2})=O(\Delta t^{m/2})$. Comparing this rate with Figure \ref{fig:local2D_RtN_SR}, it then appears that both RtN and SR solution are less smooth than the exact solution $\uu$. Furthermore, the SR solution smoothness deteriorates as the timestep is reduced.\footnote{For a visible example of what happens in practice, observe the oscillations in the SR solutions shown in Figure \ref{fig:solplot}} Note that the RtN rate here does not deteriorate since for small $\Delta t$ stagnation occurs at the initial condition (cf.~Section \ref{subsec:RtNstagnation}), which is smooth and exactly representable.

Comparing Figure \ref{fig:local2D_RtN_SR} with what is stated in Theorem \ref{th:eps_expl_impl} we note that these experiments mostly match the theoretical predictions. We observe a $\theta$ of order $O(\Delta t^p)$ with $p>1/2$ (at least initially) when the refined differencing strategy is used and $p=0$ when it is not. Furthermore, all local errors are essentially bounded as $\Delta t$ is reduced. Finally, these results also show that when the refined differencing strategy is used $\Theta^n$ is negligible with respect to $\eps^n$, which is $O(1)$ in the worst-case. We therefore expect the $\eps^n$ terms to dominate the global error.

\paragraph{Global rounding errors} We conclude the section by analysing the global rounding errors. From now on we only work with the delta form and the refined differencing strategy since these clearly give the most accurate implementation. We just saw how in this case the $\eps^n$ term dominates over $\Theta^n$ so we expect the global errors to behave according to Scenario 1) (cf.~Section \ref{sec:global}) for RtN and to Scenario 2) for SR. We monitor the following quantities:
\begin{align}
\label{eq:global_error_measures}
u^{-1}\frac{\Vert\bm{E}^N\Vert}{\Vert\widehat{\bm{U}}^N\Vert},\ \text{for RtN};\quad\quad u^{-1}\frac{\E[\Vert\bm{E}^N(\omega)\Vert^2]^{1/2}}{\Vert\bm{U}^N\Vert},\ \text{for SR},
\end{align}
where the norms used are either the infinity or the discrete $L^2$ norm. We use standard Monte Carlo sampling to estimate the expected errors in the SR case. Rather than using a fixed number of samples, we use a number which ensures that the true expected error is within $5\%$ of the estimated value with at least $95\%$ confidence. We work with both FE and BE, as well as the RK4 and RKC2 ($16$ stages) methods, which we all compare by using the same $\Delta t$ values for $\lambda$ fixed. We remark that this is not the standard way of employing RKC methods\footnote{The stability region of RKC methods grows quadratically with the number of stages, which are increased as needed when the grid size is refined \cite[]{VerwerHundsdorfer1990RKC}. Therefore, the common usage of RKC2 would be to increase the number of stages and $\lambda$ proportionally to $1/h$ for $\Delta t$ fixed.}. Nevertheless, we believe this is a good way of comparing RKC2 with the other methods in the asymptotic regime $\Delta t \rightarrow 0$.

\begin{figure}[H]
	\vspace{-0pt}
	\centering
	\begin{subfigure}{0.83\linewidth}
		\includegraphics[width=\linewidth,trim={0 0 0 2.5cm},clip]{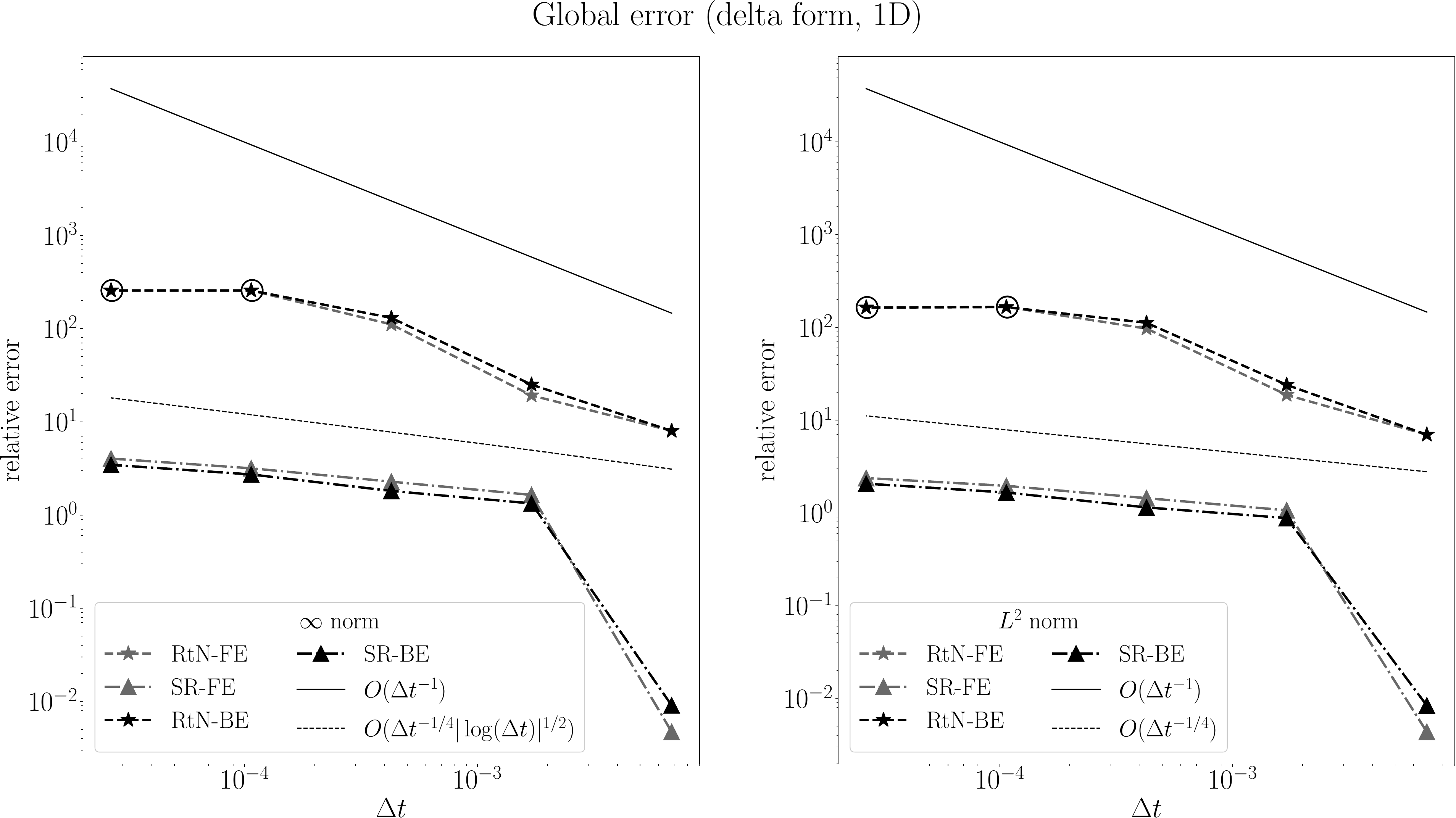}
	\end{subfigure}
	\begin{subfigure}{0.83\linewidth}
		\includegraphics[width=\linewidth,trim={0 0 0 2.5cm},clip]{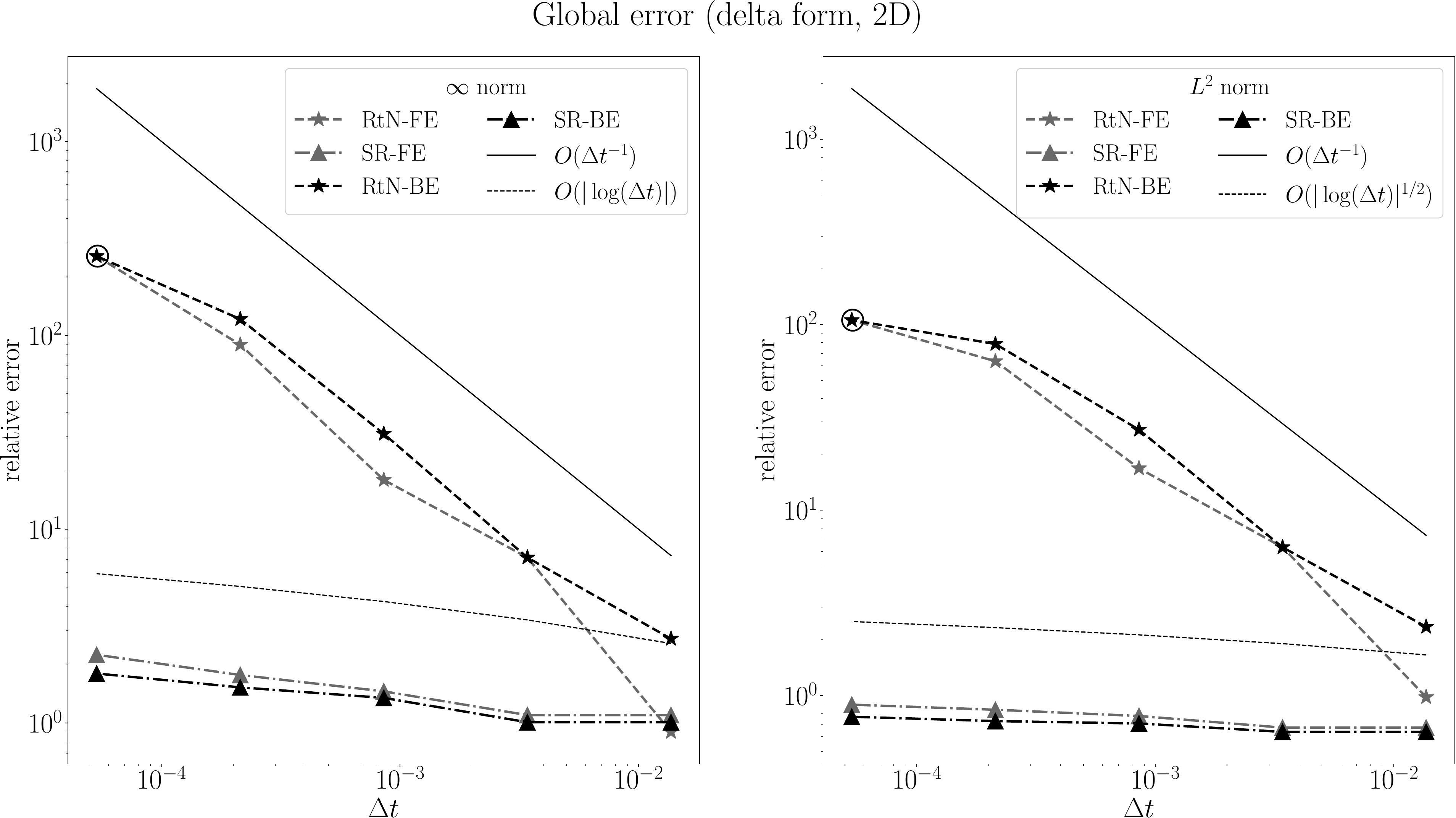}
	\end{subfigure}
	\begin{subfigure}{0.83\linewidth}
		\includegraphics[width=\linewidth,trim={0 0 0 2.5cm},clip]{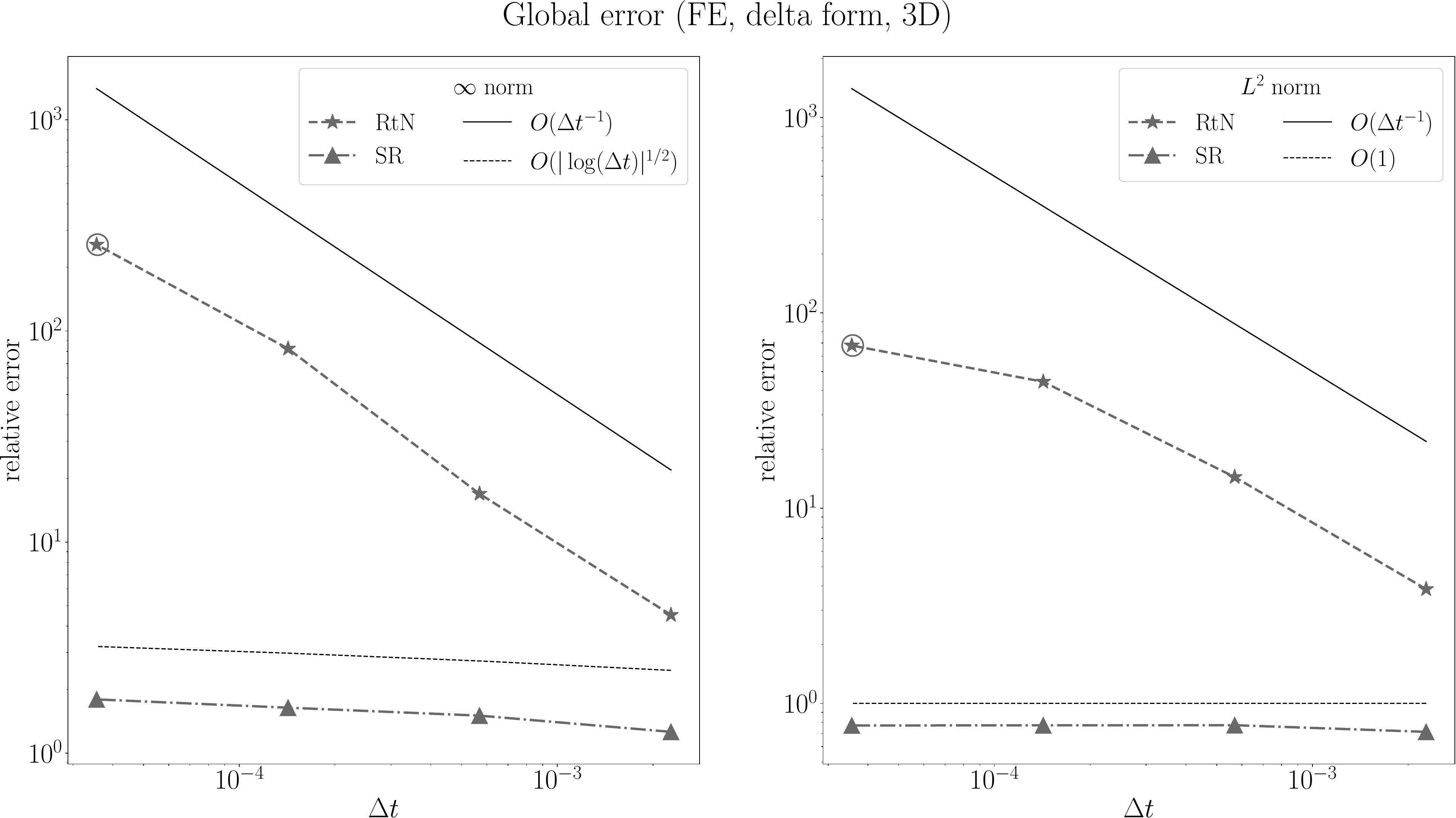}
	\end{subfigure}
	\vspace{-0pt}
	\centering
	\caption{\small\textit{Plots of the relative global rounding error in the $\infty$ (left) and discrete $L^2$ norm (right) in 1D (top), 2D (middle), and 3D (bottom). We circled the RtN data points for which the solution stagnates at the initial condition. The error behavior matches the theoretical predictions of Section \ref{sec:global}. On the coarsest mesh the 1D steady-state solution and forcing term is (almost) exactly representable, hence the low error SR data points in the 1D plots.}}
	\label{fig:global}
	\vspace{-16pt}
\end{figure}

\begin{remark}
	\label{rem:RtN_lower_upper_global_error}
	In all practical simulations with RtN, the error measures \eqref{eq:global_error_measures} always satisfied
	\begin{align}
	\label{eq:remRtNlowerupperbounds}
	u^{-1}\frac{\Vert\fl(\bm{U}^\infty) - \bm{U}^\infty\Vert}{\Vert\bm{U}^\infty\Vert} \leq \text{error measure} \leq u^{-1}\frac{\Vert\widehat{\bm{U}}^0 - \bm{U}^\infty\Vert}{\Vert\bm{U}^\infty\Vert},
	\end{align}
	where $\widehat{\bm{U}}^0=\bm{U}^0_{\bm{i}}=\uu(0,\bm{x}_{\bm{i}})$. The lower bound is actually a theoretical result that also applies to SR (after replacing $\fl(\cdot)$ with $\sr(\cdot)$ and taking the minimum over $\omega$). In fact, it indicates that the best we can hope for is to compute the discrete steady-state solution exactly and then rounding it. On the other hand, the upper bound confirms (cf.~Section \ref{subsec:RtNstagnation}) that in the worst case RtN stagnation begins from the very first timestep so that the rounded solution is stuck at the initial condition. Note that in \eqref{eq:global_error_measures} we divide the RtN error by $\Vert\widehat{\bm{U}}^N\Vert$ rather than $\Vert\bm{U}^N\Vert$, to make the comparison with SR fairer when stagnation occurs.
\end{remark}

Results are shown in Figures \ref{fig:global} and \ref{fig:global_high_order}. The SR results essentially follow the asymptotic rates derived in Section \ref{sec:global} and described in Table \ref{tab:rates}, with possibly slightly lower rates in 1D (cf.~Figure \ref{fig:global} top). These results show that SR behaves as predicted by our theory and that the $\eps^n$ term is indeed dominating the local error, thus determining the global error behavior. Comparing the SR results with the lower bound in \eqref{eq:remRtNlowerupperbounds}, we note that the $O(1)$ errors observed with SR are extremely close to the best we can theoretically hope to achieve. On the other hand, when RtN is used the relative global errors are always growing like $O(\Delta t^{-1})$ until severe stagnation occurs, matching the theoretical bounds derived in Section \ref{sec:global}. This confirms that in the RtN case rounding errors cannot be modelled as spatially or temporally uncorrelated and/or independent: if this were the case a smaller growth rate would be observed due to cancellation of cross-correlation terms in the global error analysis of Section \ref{sec:global}.

We observe that when severe stagnation occurs our RtN global error analysis becomes too pessimistic. In fact, the local errors stop accumulating in the worst possible way as the RtN solution stagnates at the initial condition (cf.~Remark \ref{rem:RtN_lower_upper_global_error}). Nevertheless, our bounds are sharp when they are needed, as the severe stagnation regime should just be avoided. We remark that for neither rounding modes we observe any additional growth in $\Delta t$ due to implicit timestepping, indicating that in this case the constants appearing in Theorem \ref{th:eps_expl_impl} might be $K$-independent. In fact, the FE and BE global rounding errors are comparable in size.

\begin{figure}[h!]
	\vspace{-0pt}
	\includegraphics[width=0.91\linewidth,trim={0 0 0 0cm},clip]{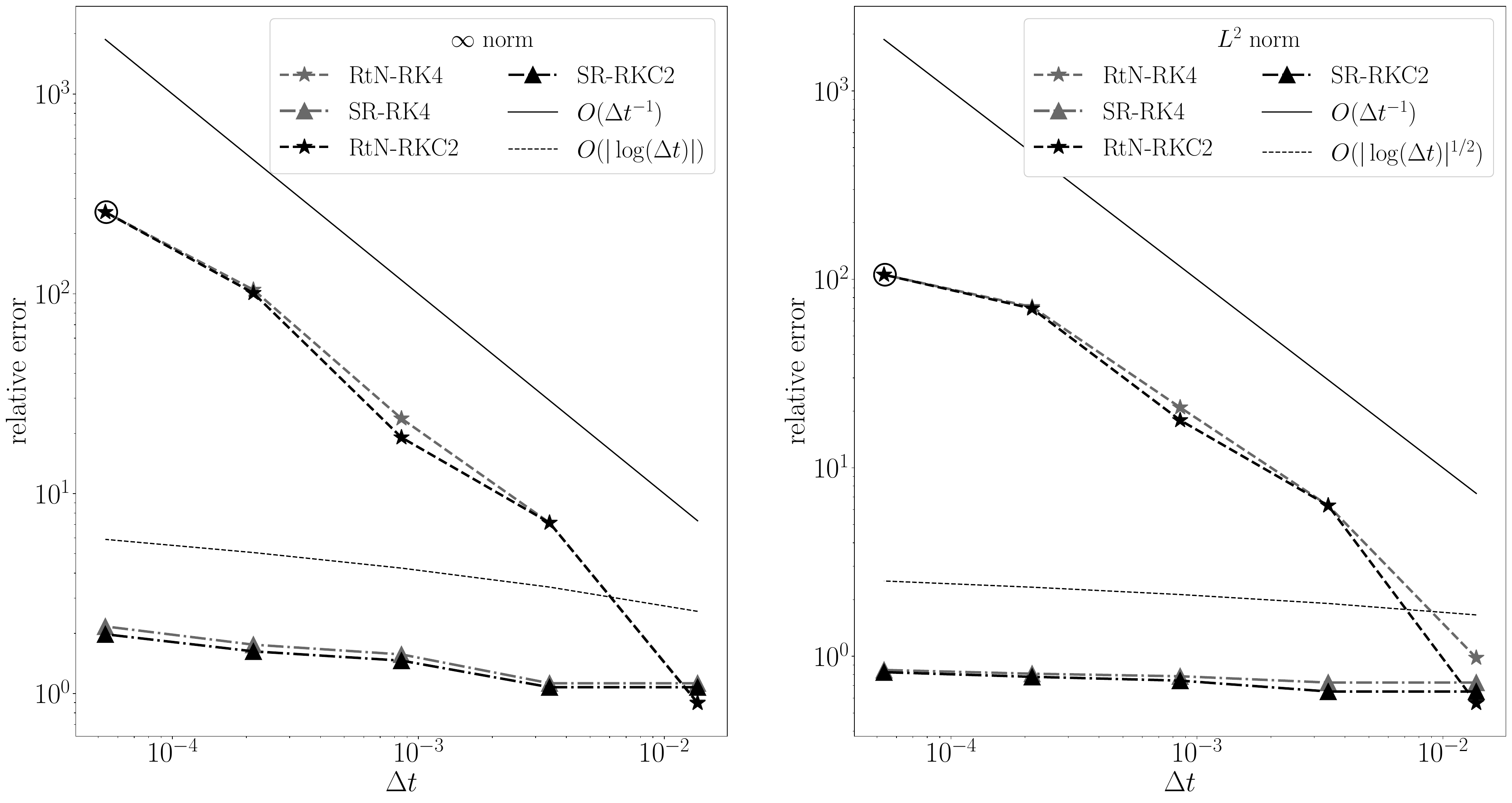}
	\vspace{-0pt}
	\centering
	\caption{\small\textit{Plots of the relative global rounding error in the $\infty$ (left) and discrete $L^2$ norm (right) in 2D for the RK4 and second-order RK-Chebyshev ($16$ stages) methods. We circled the RtN data points for which the solution stagnates at the initial condition. The global error behavior matches the theoretical predictions of Section \ref{sec:global}. Comparing these results with Figure \ref{fig:global} (middle), these higher-order methods do not seem to improve over first-order methods.}}
	\label{fig:global_high_order}
	\vspace{-6pt}
\end{figure}

In Figure \ref{fig:global_high_order} we specifically look at the global error behavior of the higher-order methods RK4 and RKC2 in 2D. Interestingly, the error behavior of high-order methods is qualitatively the same as for the first-order Euler methods. In fact, our theory holds for all RK methods, and states that no matter the order, RK methods will stop converging in reduced-precision with the error either stagnating to an $O(u)$ value (for SR in 3D) or blow-up as $\Delta t \rightarrow 0$ (with a much faster error growth rate for RtN). With these results in mind, we suspect that there might not actually be any accuracy and performance benefit of using higher-order methods in reduced precision, although they might still be required for nonlinear stability purposes.


\begin{remark}
	The SR error lines in Figure \ref{fig:global} all behave according to Scenario 2) (cf.~beginning of Section \ref{sec:global}). This suggests that the $\theta$ term behaves better than the worst case scenario (cf.~Remark \ref{rem:full_error}) and that the $\varepsilon$ term dominates. We suspect that the $\theta$ term behaves somewhere in between the two scenarios, with local rounding errors that are approximately uncorrelated (or mean-independent) across timesteps. We leave this investigation to future research.
\end{remark}

Finally, in Figure \ref{fig:1D_weird} we show the behavior of the global rounding error as a function of $\lambda=\{2^{(l+5)}+5\}_{l=0}^{l=7}$, for fixed $h=2^{-7}$ with BE in 1D. The motivation behind this experiment is that we expect the errors to worsen as the condition number of the BE linear system (which is $O(\lambda)$ in this case) grows, as mentioned in Remark \ref{rem:conditioning_lambda_blowup}. Indeed this is the case: when $\lambda\approx u^{-1}$ the rounding error starts increasing dramatically. This behavior worsens in higher dimensions where the condition number is larger and more operations are needed by the linear solver (not shown). These results suggest that $\lambda$, and consequently $\Delta t$, must be kept small even for implicit timestepping, thus restricting what is perhaps the greatest advantage of implicit methods. We suggest that a mixed-precision treatment \cite[]{abdelfattah2021survey} would be extremely beneficial here.

\begin{figure}[h!]
	\vspace{-6pt}
	\centering
	\includegraphics[width=0.9\textwidth]{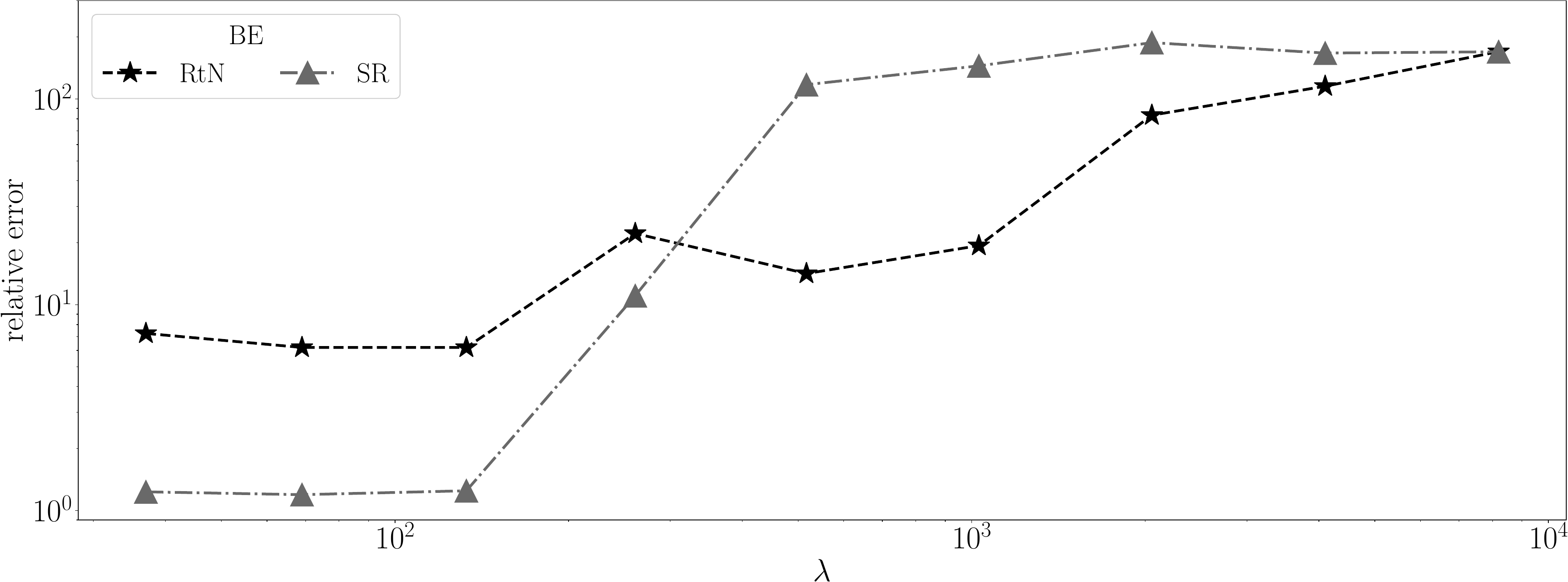}
	\caption{\textit{Plot of the relative global rounding error in the $\infty$ norm for BE in 1D as a function of $\lambda$ for fixed $h=2^{-7}$. The global error increases drastically as the conditioning of the BE linear system increases beyond roughly $u^{-1}$.}}
	\label{fig:1D_weird}
	\vspace{-24pt}
\end{figure}

%

\section{Conclusions}
\label{sec:conclusions}
We have investigated the behavior of local and global rounding errors arising from the solution of the heat equation in low precision using both round-to-nearest and stochastic rounding. We showed that a careful implementation of the scheme in delta form is essential to reduce rounding errors, and we explained why round-to-nearest computations are always affected by stagnation as $\Delta t$ decreases. For this rounding mode, the global rounding error grows like $O(u\Delta t^{-1})$ until stagnation occurs and the discrete solution converges to the initial condition for $\Delta t$ small enough. On the other hand, we showed that stochastic rounding successfully prevents stagnation and always approximates the exact solution. Furthermore, the dominant local error terms arising from stochastic rounding have favorable statistical properties, yielding a much milder global error growth rate, which is $O(u\Delta t^{-1/4})$ in 1D and roughly constant in higher dimensions. These results suggest that there might be little accuracy and performance benefit with using high-order RK schemes in reduced precision, and that a mixed-precision treatement could be more suitable for these methods.

We remark that the global error behavior described in this paper occurs regardless of the number of digits used. However, rounding errors of this type only grow appreciably when working in low precision. In this setting, our results show that stochastic rounding empowers the PDE solver with resilience to rounding errors, which enables the computation of (almost) machine-precision-accurate solutions.

While our theory and error bounds broadly also extend to linear parabolic PDEs in general (a paper on this subject is in preparation), it would be extremely interesting to analyse the effect of SR when solving other time-dipendent PDEs, especially hyperbolic or nonlinear problems. In this setting, stagnation might not necessarily occur, and SR might consequently not be as effective. We leave this investigation to future work.

\printbibliography

\begin{appendices}

\section{Proof of Theorem \ref{th:matrix_rational_function_bound}}
\label{sec:appA}
First, we need an auxiliary lemma:
\begin{lemma}
	\label{lemma:matrixpoly}
	Let $P(z)$ be a polynomial of degree $s$, let $\bm{v}\in\R^{(K-1)^d}$ and set $Y=P(-\Delta t A)$ and $\bm{y} = P(-\Delta t A)\bm{v}$. There exist constants $c_0,c_1,c_2>0$ such that ${\Vert\bm{y}\Vert_\infty \leq c_0(s,d,\lambda)\Vert\bm{v}\Vert_\infty}$ and, for $\ups$ sufficiently small, the computation of $Y$ and $\bm{y}$ in finite precision yields instead $\widehat{Y}$ and $\widehat{\bm{y}}$ satisfying
	\begin{align}
	\label{eq:lemma_appA_proposition}
	\Vert\widehat{Y}-Y\Vert_{\infty}\leq c_1(s,d,\lambda)\ups,\quad \Vert\widehat{\bm{y}}-\bm{y}\Vert_{\infty}\leq c_2(s,d,\lambda)\ups\Vert\bm{v}\Vert_\infty.
	\end{align}
	The second bound still holds, albeit with a slightly larger constant, in the case in which $\bm{v}$ is also computed inexactly yielding $\widehat{\bm{v}}$ with $\Vert\widehat{\bm{v}} - \bm{v}\Vert_\infty \leq c\ups\Vert\bm{v}\Vert_{\infty}$ for some $c>0$.
\end{lemma}

\begin{proof}
	Let $P(x)=\sum_{q=0}^sa_qx^q$ and, for simplicity of notation, let $B=-\Delta t A$, where $B$ is sparse with at most $m=3d$ nonzero entries per row. From \eqref{eq:main_eqn_FD} it is easy to see that $\Vert B\Vert_\infty=4d\lambda$. The first bound is obtained by noting that since $\bm{y}=P(B)\bm{v}$ we have
	\begin{align*}
	\Vert\bm{y}\Vert_\infty \leq \sum_{q=0}^s|a_q\Vert|B\Vert_\infty^q\Vert\bm{v}\Vert_\infty \leq c_0\Vert\bm{v}\Vert_\infty,\quad\text{where}\quad c_0=\max_q|a_q|\dfrac{(4d\lambda)^{s+1}-1}{4d\lambda - 1}.
	\end{align*}
	Now, if the evaluation of $P(B)$ is performed via explicit powers, Horner's method or the Paterson-Stockmeyer method (cf.~Section 4.2 in \cite{higham2008functions}), then Theorem 4.5 in \cite{higham2008functions} (see Section 5 of \cite{higham2002accuracy} for a proof) gives, for some $\bar{c}>0$ depending on $s$ and $m$,
	\begin{align*}
	\Vert\widehat{Y}-Y\Vert_{\infty}\leq \bar{c}\ups\sum_{q=0}^s|a_q\Vert|B\Vert_\infty^q\leq c_1(s,d,\lambda)\ups,
	\end{align*}
	where $c_1(s,d,\lambda)=\bar{c}c_0$. Note that $\bar{c}$ here does not depend on the size of $B$ since $B$ is sparse. When $\bm{v}$ is computed exactly, the other bound in \eqref{eq:lemma_appA_proposition} is obtained with the same argument by using the classical error bound for sparse matrix-vector multiplication (see Section 2 in \cite{higham2002accuracy}):
	\begin{align*}
	\widehat{B\bm{v}}=(B+\Delta B)\bm{v},\quad\text{with}\quad |\Delta B|\leq \gamma_m|B|.
	\end{align*}
	When $\bm{v}$ is evaluated inexactly we can apply the bound for the exact vector evaluation to $\widehat{\bm{v}}$, giving
	\begin{align*}
	\Vert\widehat{\bm{y}}-\bm{y}\Vert_{\infty} \leq c_2\ups\Vert\widehat{\bm{v}}\Vert_\infty \leq c_2\ups(1+c\ups)\Vert\bm{v}\Vert_\infty,
	\end{align*}
	and the proposition follows after renaming the constant in the bound.
\end{proof}

We can now prove Theorem \ref{th:matrix_rational_function_bound}.

\begin{proof}[Proof of Theorem \ref{th:matrix_rational_function_bound}]
	Lemma \ref{lemma:matrixpoly} already gives the result for explicit RK methods. For implicit methods, let $B=-\Delta t A$. We evaluate $\bm{z}=P(B)\bm{v}$ by solving the linear system $W(B)\bm{z}=N(B)\bm{v}=\bm{y}$. Since $W(x)$ has no roots on the negative real axis and $B$ only has negative eigenvalues, $W(B)$ is invertible and there exists $\alpha>0$ such that $\Vert W(B)^{-1}\Vert_\infty\leq \alpha(s,d,\lambda,K)$. Thanks to Lemma \ref{lemma:matrixpoly} we then obtain $\Vert\bm{z}\Vert_\infty\leq \alpha c_0\Vert\bm{v}\Vert_\infty$, which is the first preposition. Now, due to Lemma \ref{lemma:matrixpoly}, the evaluation of $\bm{z}$ in finite precision yields
	\begin{align*}
	(W(B)+\Delta W)\widehat{\bm{z}}=\bm{y}+\Delta \bm{y},\qquad \Vert\Delta W\Vert_\infty\leq c_1\ups,\quad \Vert\Delta \bm{y}\Vert_\infty \leq c_2\ups\Vert\bm{v}\Vert_\infty
	\end{align*}
	where $\widehat{\bm{z}}$ comes from the finite-precision solution of the perturbed linear system. Since the system is solved by a stable method by assumption, we can combine the first preposition of the theorem with Theorem 7.2 in \cite{higham2002accuracy} to obtain, for some constant $\bar{c}>0$,
	\begin{align*}
	\Vert\widehat{\bm{z}}-\bm{z}\Vert_\infty \leq \bar{c}(s,d,\lambda,K)\ups \left(c_1\alpha + c_2\alpha^2c_0\right) \Vert\bm{v}\Vert_\infty=C(s,d,\lambda,K)\nu\Vert\bm{v}\Vert_\infty.
	\end{align*}
	This is the proposition that we wanted to prove. When $\bm{v}$ is evaluated inexactly we can apply the bound for the exact vector evaluation to $\widehat{\bm{v}}$ and then use the triangle inequality in the same way as in the proof of Lemma \ref{lemma:matrixpoly}.
\end{proof}

\section{Proof of Theorem \ref{thm:Sseriesbounds}}
\label{sec:appB}
We first need two auxiliary results:
\begin{lemma}
	\label{lemma:eigenbounds}
	The $s_{\bm{k}}$ and $\bm{v}^h_{\bm{k}}$ all satisfy
	\begin{align}
	-4d \lambda \leq s_{\bm{k}} &\leq -d\pi^2\Delta t\left(1-\frac{\pi^2}{12}\lambda^{-1}\Delta t\right) \leq 0,\\
	-\Delta t\pi^2\Vert\bm{k}\Vert_2^2 \leq s_{\bm{k}} &\leq  -\left(1-\frac{\pi^2}{12}\right)\Delta t\pi^2\Vert\bm{k}\Vert_2^2,
	\label{eq:s_k_bounds}
	\\
	0\leq (\bm{1},\bm{v}^h_{\bm{k}})&=\prod_{j=1}^d\cot\left(\frac{\pi \bm{k}_j}{2K}\right)(\bm{k}_j\hspace{-8pt}\mod 2)\leq 2^d\pi^{-d}K^d\prod_{j=1}^d\frac{1}{\bm{k}_j}(\bm{k}_j\hspace{-8pt}\mod 2).
	\label{eq:cotan_vk}
	\end{align}
	where the upper and lower bound of the first inequality are only attained in the limit as $\Delta t,h\rightarrow 0$ and strict inequalities hold otherwise.
\end{lemma}

\begin{proof}
	The relations involving the $s_{\bm{k}}$ terms are obtained by using the inequality $x^2 - x^4/3 \leq \sin(x)^2 \leq \min(1,x^2)$, valid for all $x\in \R$, to \eqref{eq:eigenpairsA}. The definition of $\bm{v}^h_{\bm{k}}$ in \eqref{eq:eigenpairsA} gives
	\begin{align}
	(\bm{1},\bm{v}^h_{\bm{k}}) = \prod\limits_{j=1}^d\left(\sum\limits_{\bi_j=1}^{K-1}\sin\left(\pi\bi_j\frac{\bm{k}_j}{K}\right)\right).
	\end{align}
	Applying Lagrange trigonometric identity to the series of sines yields the equality in \eqref{eq:cotan_vk}. The upper bound is obtained by using the relation $\cot(x)\leq 1/x$, valid for all $x\in(0,\pi)$.
\end{proof}

\begin{lemma}
	\label{lemma:bounds_series_oneovernormsquared}
	The following bounds hold for $K\geq 1$ and $\bm{k}\in\N^d$:
	\begin{align}
	\label{eq:lemmaboundsseries1}
	\sum\limits_{\bm{k}=\bm{1}}^{\bm{k}=K\bm{1}}\frac{1}{\Vert\bm{k}\Vert_2^2}&\leq\left\{
	\begin{array}{lr}
	\frac{\pi^2}{6}, & d=1,\\
	\frac{\pi}{2}\log(\sqrt{2}K) + \frac{1}{2}, & d=2,\\
	\frac{\pi\sqrt{3}}{2}K -\frac{\pi}{2} + \frac{1}{3},&d=3,
	\end{array}
	\right.\\
	\sum\limits_{\bm{k}=\bm{1}}^{\bm{k}=K\bm{1}}\frac{1}{\Vert\bm{k}\Vert_2^4}&\leq \left\{
	\begin{array}{lr}
	\frac{\pi^4}{90}, & d=1,\\
	\frac{\pi}{4} + \frac{1}{4}, & d=2,\\
	\frac{\pi}{2} + \frac{1}{9}, & d=3.
	\end{array}
	\right.
	\label{eq:lemmaboundsseries2}
	\end{align}
\end{lemma}

\begin{proof}
	The 1D result is actually an equality and is well-known. Let $r=\Vert\bm{x}\Vert_2$. If $f(r)$ is strictly positive and monotonically decreasing on $[1,\infty)$, then, in 2D
	\begin{align*}
	f(\sqrt{i^2+j^2})&<\int_{i-1}^i\int_{j-1}^j f(\sqrt{x^2+y^2})\, \text{d} x \text{d} y,
	\\
	\Longrightarrow\quad\sum_{\bm{k}=\bm{1}}^{\bm{k}=K\bm{1}} f(\Vert\bm{k}\Vert) &\leq f(\sqrt{2}) + \int_1^{\sqrt{2}K} \frac{\pi}{2} r\, f(r)\, \text{d} r,
	\end{align*}
	while in 3D
	\begin{align*}
	f(\sqrt{i^2+j^2+k^2})&<\int_{i-1}^i\int_{j-1}^j\int_{k-1}^k f(\sqrt{x^2+y^2+z^2})\, \text{d} x \text{d} y \text{d} z,
	\\
	\Longrightarrow\quad\sum_{\bm{k}=\bm{1}}^{\bm{k}=K\bm{1}} f(\Vert\bm{k}\Vert) &\leq f(\sqrt{3}) + \int_1^{\sqrt{3}K} \frac{\pi}{2} r^2\, f(r)\, \text{d} r.
	\end{align*}
	The desired bounds then follow naturally using $f(r)\equiv r^{-2}$ for \eqref{eq:lemmaboundsseries1} and
	$f(r)\equiv r^{-4}$ for \eqref{eq:lemmaboundsseries2}.
\end{proof}

Finally, we can prove Theorem \ref{thm:Sseriesbounds}:

\begin{proof}
	We start by proving \eqref{eq:Sseriesbounds1}. Since $(1-x^2)^{-1}\leq (1-|x|)^{-1}$ for all $x$ we just need to prove the second inequality in \eqref{eq:Sseriesbounds1}. Since $S(z)=1+z+O(z^2)$ and $\left(\Lambda \cup \{-4 d \lambda\}\right) \subset \mathcal{A}$, there exist $s_0,\alpha_0>0$ such that
	\begin{align*}
	0 \leq S(z) \leq 1+\frac{z}{2},\ \text{ for }z\in(-s_0,0),\ \text{ and }\ \alpha_0 \leq 1-|S(z)| \leq 1,\ \text{ for } z\in[-4 d \lambda,-s_0].
	\end{align*}
	We can therefore split
	\begin{align}
	\label{eq:temp_1overonemS}
	\sum_{\bm{k}}\dfrac{1}{1-|S(s_{\bm{k}})|}=\sum\limits_{|s_{\bm{k}}|\leq s_0}\dfrac{1}{1-|S(s_{\bm{k}})|} + \sum\limits_{|s_{\bm{k}}|> s_0}\dfrac{1}{1-|S(s_{\bm{k}})|}\leq \sum\limits_{|s_{\bm{k}}|\leq s_0}\frac{2}{-s_{\bm{k}}} + \dfrac{(K-1)^d}{\alpha_0},
	\end{align}
	Due to \eqref{eq:s_k_bounds}, and since $-s_{\bm{k}}> 0$ for all $\bm{k}$, we now have
	\begin{align}
	\sum\limits_{|s_{\bm{k}}|\leq s_0}\frac{1}{-s_{\bm{k}}}\leq \frac{1}{c\Delta t}\sum_{\bm{k}}\frac{1}{\Vert\bm{k}\Vert_2^2},\quad\text{where}\quad c=\pi^2\left(1-\frac{\pi^2}{12}\right).
	\end{align}
	By combining the above with \eqref{eq:temp_1overonemS} and Lemma \ref{lemma:bounds_series_oneovernormsquared} we obtain \eqref{eq:Sseriesbounds1} since
	\begin{align*}
	\sum_{\bm{k}}\dfrac{1}{1-|S(s_{\bm{k}})|}&\leq \frac{2 \Delta t^{-1}}{c}\sum_{\bm{k}}\frac{1}{\Vert\bm{k}\Vert_2^2}+\frac{(K-1)^d}{\alpha_0}\leq \phi_d(\Delta t).
	\end{align*}
	To prove \eqref{eq:Sseriesbounds2} the reasoning is the same yielding this time,
	\begin{align}
	\sum_{\bm{k}}\dfrac{1}{(1-|S(s_{\bm{k}})|)^2}\leq \frac{4}{c^2\Delta t^2}\sum_{\bm{k}}\frac{1}{\Vert\bm{k}\Vert_2^4} + \dfrac{(K-1)^d}{\alpha_0^2}\leq \bar{C}(d,\lambda)^2\Delta t^{-2},
	\end{align}
	cf.~Lemma \ref{lemma:bounds_series_oneovernormsquared}. To prove \eqref{eq:Sseriesbounds3}, we also follow the same strategy, giving
	\begin{align}
	\label{eq:temp_theorem_vkS}
	K^{-d}\sum_{\bm{k}}\dfrac{(\bm{1},\bm{v}^h_{\bm{k}})}{1-|S(s_{\bm{k}})|}\leq \frac{K^{-d}}{c\Delta t}\sum_{\bm{k}}\frac{(\bm{1},\bm{v}^h_{\bm{k}})}{\Vert\bm{k}\Vert_2^2} + \frac{K^{-d}}{\alpha_0}\sum_{\bm{k}}(\bm{1},\bm{v}^h_{\bm{k}}).
	\end{align}
	By the Cauchy-Schwarz inequality, Lemma \ref{lemma:eigenbounds} and Lemma \ref{lemma:bounds_series_oneovernormsquared} we have
	\begin{align*}
	\sum_{\bm{k}}\frac{(\bm{1},\bm{v}^h_{\bm{k}})}{\Vert\bm{k}\Vert_2^2} \leq \left(\sum_{\bm{k}}(\bm{1},\bm{v}^h_{\bm{k}})^2\right)^{1/2}\left(\sum_{\bm{k}}\frac{1}{\Vert\bm{k}\Vert_2^4}\right)^{1/2} &\leq 2^d\pi^{-d}K^d\left(\sum_{n\text{ odd}}\frac{1}{n^2}\right)^{d/2}\hspace{-9pt}\sqrt{\tilde{c}(d)}
	\\
	& < 2^{-d/2}K^d\sqrt{\tilde{c}(d)},
	\end{align*}
	where we have used the fact that the sum of the inverse squares of all odd numbers is $\pi^2/8$.
	For the second summation in \eqref{eq:temp_theorem_vkS} we get owing to Lemma \ref{lemma:eigenbounds},
	\begin{align*}
	\sum_{\bm{k}}(\bm{1},\bm{v}^h_{\bm{k}}) \leq 2^d\pi^{-d}K^d\left(\sum_{n\text{ odd}}\frac{1}{n}\right)^d\leq \pi^{-d}K^d\left(\log(K-1)+2\right)^d,
	\end{align*}
	where we used the fact that the harmonic series over the first $K-1$ odd numbers is bounded by ${\log(K-1)/2+1}$. The two bounds for the summations and \eqref{eq:temp_theorem_vkS} give
	\begin{align*}
	K^{-d}\sum_{\bm{k}}\dfrac{(\bm{1},\bm{v}^h_{\bm{k}})}{1-|S(s_{\bm{k}})|}\leq 2^{-d/2}c^{-1}\sqrt{\tilde{c}}\Delta t^{-1} + \pi^{-d}(\log(K-1)+2)^d \leq \tilde{C}(d,\lambda)\Delta t^{-1},
	\end{align*}
	which is \eqref{eq:Sseriesbounds3} and concludes the proof.
\end{proof}
\end{appendices}
\end{document}